\documentclass[reqno,11pt]{amsart}
\usepackage[letterpaper, margin=1in]{geometry}
\RequirePackage{amsmath,amssymb,amsthm,graphicx,mathrsfs,url,slashed,subcaption,empheq}
\RequirePackage[usenames,dvipsnames]{xcolor}
\RequirePackage[colorlinks=true,linkcolor=Red,citecolor=Green]{hyperref}
\RequirePackage{amsxtra}
\usepackage{cancel}
\usepackage{tikz-cd, wrapfig, enumitem}

\makeatletter
\@namedef{subjclassname@2020}{\textup{2020} Mathematics Subject Classification}
\makeatother

\title{The canonical lamination calibrated by a cohomology class}
\author{Aidan Backus}
\address{Department of Mathematics, University of Toronto}
\email{aidan.backus@utoronto.ca}
\date{\today}
\thanks{}

\newcommand{\NN}{\mathbf{N}}
\newcommand{\ZZ}{\mathbf{Z}}
\newcommand{\QQ}{\mathbf{Q}}
\newcommand{\RR}{\mathbf{R}}
\newcommand{\CC}{\mathbf{C}}

\newcommand{\PP}{\mathbf P}

\newcommand{\Sph}{\mathbf S}

\newcommand{\Ball}{\mathbf{B}}

\newcommand*\dif{\mathop{}\!\mathrm{d}}

\DeclareMathOperator{\card}{card}
\DeclareMathOperator{\dist}{dist}
\DeclareMathOperator{\id}{id}
\DeclareMathOperator{\Hom}{Hom}

\DeclareMathOperator{\PD}{PD}
\DeclareMathOperator{\supp}{supp}

\DeclareMathOperator{\tr}{tr}

\newcommand{\Two}{\mathrm{I\!I}}
\newcommand{\Ric}{\mathrm{Ric}}

\newcommand{\weakto}{\rightharpoonup}

\newcommand{\normal}{\mathbf n}

\newcommand{\vol}{\mathrm{vol}}

\newcommand{\diam}{\mathrm{diam}}
\DeclareMathOperator{\Gal}{Gal}
\DeclareMathOperator{\arcosh}{arcosh}

\newcommand{\Lip}{\mathrm{Lip}}

\newcommand{\Riem}{\mathrm{Riem}}

\newcommand{\Mass}{\mathbf M}

\newcommand{\dfn}[1]{\emph{#1}\index{#1}}

\newcommand{\loc}{\mathrm{loc}}
\newcommand{\cpt}{\mathrm{cpt}}

\newtheorem{theorem}{Theorem}[section]

\newtheorem{lemma}[theorem]{Lemma}

\newtheorem{corollary}[theorem]{Corollary}
\newtheorem{conjecture}[theorem]{Conjecture}

\theoremstyle{definition}
\newtheorem{definition}[theorem]{Definition}

\newtheorem{example}[theorem]{Example}

\newtheorem*{ack}{Acknowledgments}



\newlist{prfenum}{enumerate}{1}
\setlist[prfenum]{
  nosep,
  label=(\theprfenum--\arabic*),
}
\newcounter{prfenum}
\AddToHook{env/proof/begin}{\setcounter{prfenum}{0}}
\AddToHook{env/prfenum/begin}{\stepcounter{prfenum}}

\numberwithin{equation}{section}

\def\XXint#1#2#3{{\setbox0=\hbox{$#1{#2#3}{\int}$ }
\vcenter{\hbox{$#2#3$ }}\kern-.6\wd0}}

\usepackage[backend=bibtex,style=alphabetic,giveninits=true]{biblatex}

\renewbibmacro{in:}{}
\DeclareFieldFormat{pages}{#1}

\begin{document}
\begin{abstract}
Let $\rho$ be a unit cohomology class of degree $d - 1$, on a closed oriented Riemannian manifold of dimension $2 \leq d \leq 7$. We construct a lamination $\lambda_\rho$ whose leaves are exactly the minimal hypersurfaces calibrated by every calibration in $\rho$. The geometry of $\lambda_\rho$ is closely related to the geometry of the unit ball of $H_{d - 1}(M, \mathbb R)$ when it is equipped with Gromov's stable norm, so our main theorem constrains the shape of the stable unit ball in terms of the topology of $M$. These results establish a close analogy between the stable norm and Thurston's earthquake norm on the tangent space to Teichmüller space.
\end{abstract}

\maketitle

\section{Introduction}
Let $M$ be a closed oriented Riemannian manifold of dimension $2 \leq d \leq 7$.
The \dfn{stable norm} $\|\alpha\|_1$ of a homology class $\alpha \in H_{d - 1}(M, \RR)$ is the infimum of the area of all $d - 1$-cycles representing $\alpha$.
The stable norm was introduced by Federer in his work \cite{Federer1974}, on the duality between area-minimizing currents and calibration cochains.
Among other applications, the stable norm was studied by Gromov, \cite{gromov2007metric}, for its connections to systolic geometry and Brock and Dunfield, \cite{Brock2017}, because of its connection to the Thurston--Gromov simplicial norm.

A quarter-century ago, Auer and Bangert released a research announcement \cite{Auer01}, which proposed to study codimension-$1$ measured oriented laminations $\lambda$ in $M$ which minimize their mass in their homology class $[\lambda] \in H_{d - 1}(M, \RR)$.\footnote{We review the basic definitions related to laminations in \S\ref{sec: lamination review}. In our convention, \emph{all laminations are Lipschitz and nonempty}.}\footnote{Some of Auer and Bangert's work appears in an unpublished manuscript, \cite{Auer12}.}
In codimension $1$, every homology class $\alpha$ can be represented by a mass-minimizing lamination $\lambda$ (whose mass then equals $\|\alpha\|_1$), which one can think of roughly think of as a canonical choice of representative of $\alpha$.
If two laminations have common leaves, those leaves cannot intersect, and Auer and Bangert proposed to use this observation to establish a deep connection between the intersection theory of $M$ and the geometry of the stable unit ball.
A similar approach was used by Balacheff and Massart, \cite{Massart96, Massart2007}, to study the stable unit ball when $M$ is a negatively curved surface. 

While trying to prove the theorems claimed in \cite{Auer01}, it is often useful to imitate ideas of the works \cite{Thurston98,Wolpert82,Gu_ritaud_2017,huang2024earthquakemetricteichmullerspace} of the Thurston school on Thurston's asymmetric metric on Teichm\"uller space.
To make this precise, let $g \geq 2$, let $\Sigma_g$ be the closed oriented surface of genus $g$, let $\mathscr T_g$ be its Teichm\"uller space, let $\rho, \sigma \in \mathscr T_g$ be hyperbolic metrics on $\Sigma_g$, and let $L(\rho, \sigma) \geq 1$ be the infimum of Lipschitz constants of maps $(\Sigma_g, \rho) \to (\Sigma_g, \sigma)$ homotopic to $\id_{\Sigma_g}$.
\dfn{Thurston's stretch metric} on $\mathscr T_g$ is $\log L$.
Thurston's stretch metric is studied using geodesic laminations on $(\Sigma_g, \rho)$, and in particular the \dfn{canonical maximally-stretched lamination} given by the following theorem.

\begin{theorem}[{\cite{Gu_ritaud_2017}}]\label{existence of stretched lam}
For every $g \geq 2$ and $\rho, \sigma \in \mathscr T_g$, there exists a unique largest chain-recurrent geodesic lamination $\lambda_{\rho, \sigma}$ in $(\Sigma_g, \rho)$ such that for every Lipschitz map $f: (\Sigma_g, \rho) \to (\Sigma_g, \sigma)$ homotopic to $\id_{\Sigma_g}$ such that $\Lip(f) = L(\rho, \sigma)$, $f$ stretches every leaf of $\lambda_{\rho, \sigma}$ by a factor of $L(\rho, \sigma)$.
\end{theorem}

The \dfn{costable norm}, $\|\cdot\|_\infty$, on $H^{d - 1}(M, \RR)$ is the dual norm of the stable norm.
The purpose of this paper is to flesh out the idea that the costable norm of $M$ is closely analogous to Thurston's stretch distance in an infinitesimal neighborhood of a hyperbolic metric.
Our main theorem is that for each class of unit costable norm, there is a lamination in $M$ of minimal hypersurfaces which is analogous to Thurston's canonical lamination.
Studying the structure of this lamination yields several of the proposed theorems of \cite{Auer01}.
A sample application of the theory we shall develop is that under purely topological assumptions, $M$ has many uniquely ergodic laminations of minimal hypersurfaces.

A \dfn{calibration} (of codimension $1$) is a closed $d - 1$-form $F$ such that $\|F\|_{L^\infty} = 1$.
For any calibration $F$ on $M$, a hypersurface $N \subset M$ is \dfn{$F$-calibrated} if the pullback of $F$ to $N$ is the area form on $N$.
If $N$ is $F$-calibrated, then the mean curvature of $N$ is $0$, and if $N$ is closed then $N$ minimizes its area in its real homology class.
This brings us to our main theorem:

\begin{theorem}\label{existence of calibrated lam}
For every $\rho \in H^{d - 1}(M, \RR)$ such that $\|\rho\|_\infty = 1$, there is a unique largest lamination $\lambda_\rho$ in $M$ such that for every calibration $F$ representing $\rho$, $F$ calibrates every leaf of $\lambda_\rho$.
\end{theorem}

The lamination $\lambda_\rho$ is the \dfn{canonical lamination} calibrated by $\rho$.
We prove Theorem \ref{existence of calibrated lam} in \S\ref{canonical sec}.
The proof uses multiple results from \cite{BackusCML}, including the interpretation of mass-minimizing laminations in terms of \dfn{functions of least gradient}, functions $u$ which minimize their total variation $\int_M \star |\dif u|$.
We carefully note that the lamination $\lambda_\rho$ is not itself mass-minimizing, since it may not admit any sort of transverse measure and therefore does not have a well-defined homology class.
However, any measured sublamination of $\lambda_\rho$ is mass-minimizing in its homology class.

To illustrate Theorem \ref{existence of calibrated lam}, suppose that $d = 2$, so that we can identify homotopy classes of maps $M \to \Sph^1$ with homomorphisms $\pi_1(M) \to \ZZ$, which in turn can be identified with lattice points in $H^1(M, \RR)$.
Let $\rho$ be such a lattice point; by rescaling $M$ we may assume that $\|\rho\|_\infty = 1$.
In that case, any calibration which represents $\rho$ is the derivative of a minimizing Lipschitz map in the homotopy class $\rho$, and every leaf of the canonical lamination calibrated by $\rho$ is maximally stretched by every minimizing Lipschitz map in $\rho$. 
Thus Theorem \ref{existence of calibrated lam} is a generalization of a version of Theorem \ref{existence of stretched lam} where one works with homotopy classes of maps $M \to \Sph^1$ rather than $[\id_{\Sigma_g}]$.

See \S\ref{sec: high dimension} for a discussion of the possible generalizations of Theorem \ref{existence of calibrated lam} to higher dimension and codimension.

In \S\ref{canonical structure}, we study measured sublaminations of canonical calibrated laminations.
This is motivated by the fact that the \dfn{earthquake norm}, the dual of the norm induced by Thurston's stretch metric, is not strictly convex, and its failure to be strictly convex detects the failure of geodesic laminations to be uniquely ergodic \cite{huang2024earthquakemetricteichmullerspace}.
This suggests that if the stable norm is not strictly convex, then there should be canonical calibrated laminations which are not uniquely ergodic; it turns out that this is exactly what happens.

A finite Borel measure $\mu$ is \dfn{transverse} to a lamination $\lambda$, if $\supp \mu = \supp \lambda$ and $\mu$ is invariant under deformations which preserve the area forms of every leaf of $\lambda$.
A transverse probability measure $\mu$ is \dfn{ergodic} if, for every Borel set $E$ which is a union of leaves of $\lambda$, either $\mu(E) = 0$ or $\mu(E) = 1$.
The lamination $\lambda$ is \dfn{uniquely ergodic}, if there is a unique probability measure $\mu$ which is transverse to $\lambda$ (in which case $\mu$ must be ergodic).

Let $\rho \in H^{d - 1}(M, \RR)$ have unit norm.
The canonical lamination $\lambda_\rho$ may not admit a transverse measure, but the proof of Theorem \ref{existence of calibrated lam} shows that $\lambda_\rho$ has a sublamination which admits an ergodic transverse measure.
Let
$$B := \{\alpha \in H_{d - 1}(M, \RR): \|\alpha\|_1 \leq 1\}$$
be the stable unit ball, and let
$$\rho^* := \{\alpha \in \partial B: \langle \rho, \alpha\rangle = 1\}$$
be the dual flat to $\rho$.
Since the stable norm does not have to be convex, $\rho^*$ does not have to be a singleton.

\begin{corollary}\label{extreme points are indecomposable}
For every $\rho \in H^{d - 1}(M, \RR)$ with $\|\rho\|_\infty = 1$, $\rho^*$ is the set of homology classes which are represented by probability measures which are transverse to sublaminations of the canonical lamination $\lambda_\rho$.
Every extreme point of $\rho^*$ is represented by an ergodic measure on a sublamination of $\lambda_\rho$.
\end{corollary}

Auer and Bangert \cite{Auer01} observed that one can use a lemma of Arnoux and Levitt \cite{Arnoux1986} to estimate the number of ergodic measures on sublaminations of a lamination without closed leaves.
So by Corollary \ref{extreme points are indecomposable}, the Arnoux--Levitt lemma applied to $\lambda_\rho$ determines the structure of $\rho^*$.
A homology class $\alpha \in H_{d - 1}(M, \RR)$ has \dfn{rational direction} if there exists $c > 0$ such that $c\alpha$ is in the image of the map $H_{d - 1}(M, \ZZ) \to H_{d - 1}(M, \RR)$.
Let $b_1(M) := \dim H_1(M, \QQ)$ be the first Betti number.

\begin{theorem}\label{vertex counting}
Let $F$ be a maximal flat of the stable unit sphere $\partial B$. Then:
\begin{enumerate}
\item $F$ is a convex polytope.
\item The number of vertices of $F$ with irrational direction is at most $\max(0, b_1(M) - 1)$.
\item A vertex $\alpha$ of $F$ has rational direction iff $\alpha$ is represented by a closed leaf of $\lambda_\rho$.
\end{enumerate}
\end{theorem}

We illustrate this theorem with two examples where $\partial B$ cannot be strictly convex.
\begin{enumerate}
\item Suppose that $M$ is homeomorphic to the closed oriented surface $\Sigma_g$ of genus $g \geq 2$.
By a theorem of Massart \cite{Massart1997StableNO}, if a maximal flat $\rho^*$ has a point of rational direction, then $\rho^*$ has at most $3g - 3$ vertices, all of which have rational direction.
It follows that $\lambda_\rho$ consists of at most $3g - 3$ closed geodesics, plus possibly a ``spiraling'' part which admits no transverse measures.
On the other hand, if $\rho^*$ has no points of rational direction, then $\rho^*$ has at most $2g - 1$ vertices, and $\lambda_\rho$ has no closed leaves.
\item Suppose that $d = 3$, $M$ has positive scalar curvature $R \geq R_{\rm min} > 0$, and $\rho^*$ is a maximal flat.
The inradius of any stable minimal surface $N \subset M$ is $\lesssim R_{\rm min}^{-1/2}$ \cite[Proposition 2.2]{Liokumovich18}; this holds when $N$ is an intrinsic ball in a leaf of $\lambda_\rho$, since calibrated surfaces are stable minimal. 
So every leaf of $\lambda_\rho$ must be closed, and every vertex of $\rho^*$ must have rational direction.
\end{enumerate}

\begin{corollary}\label{strict convexity iff unique ergodicity}
If the stable unit ball $B$ is strictly convex, then every ergodic calibrated lamination is uniquely ergodic.
In particular, if $b_1(M) \geq 2$ and $B$ is strictly convex, then all but countably many homology classes in $\partial B$ are represented by uniquely ergodic calibrated laminations without closed leaves.
\end{corollary}

The analysis of $\lambda_\rho$ yields the following theorem on the strict convexity of the stable unit ball, which was proposed without proof by Auer and Bangert \cite[Theorems 6 and 7]{Auer01}.
The \dfn{intersection product} $\alpha \cdot \beta$ of two homology classes $\alpha, \beta$ is the Poincar\'e dual of the cup product $\PD(\alpha) \smile \PD(\beta)$, and the \dfn{derived series} of a group $\Gamma$ is defined by letting $\Gamma^{(0)} := \Gamma$ and $\Gamma^{(n + 1)}$ be the commutator subgroup of $\Gamma^{(n)}$.

\begin{theorem}\label{Auer Bangert thm}
One has:
\begin{enumerate}
\item If there is a line segment $[\alpha, \beta] \subset \partial B$, then $\alpha \cdot \beta = 0$. \label{ABt1}
\item Let $\Gamma := \pi_1(M)$. If $\Gamma^{(1)}/\Gamma^{(2)}$ is a torsion group, then $B$ is strictly convex. \label{ABt2}
\end{enumerate}
\end{theorem}

For example, suppose that $M$ has the homotopy type of a torus.
Then $B$ is strictly convex and so:
\begin{enumerate}
\item $M$ has many uniquely ergodic laminations of minimal hypersurfaces.
If $M = \RR^d/[0, 1]^d$, then these laminations are the irrational foliations of $M$, but what Theorem \ref{Auer Bangert thm} says is that the same behavior occurs regardless of the Riemannian metric on $M$.
\item If $d = 3$, then the discussion after Theorem \ref{vertex counting} implies that $M$ does not have positive scalar curvature. This is a classic theorem of Schoen and Yau \cite{Schoen1979} which has a much more direct proof.
\end{enumerate}

Since the canonical calibrated lamination $\lambda_\rho$ does not have to be uniquely ergodic, it is natural to ask if there is a canonical measure on $\lambda_\rho$.
In upcoming work \cite{daskalopoulos2025}, Daskalopoulos and Uhlenbeck will show there is a favored measure on any canonical maximally stretched lamination with only closed leaves.
The same proof works on $\lambda_\rho$, but it is quite lengthy, so we omit the proof.
For each $p < \infty$, there is a unique representative $F_p \in L^p(M, \Omega^{d - 1})$ of $\rho$ which is \dfn{$p$-harmonic}; that is,
$$\dif F_p = 0, \quad \dif^*(|F_p|^{p - 2} F_p) = 0.$$

\begin{theorem}[{\cite{daskalopoulos2025}}]\label{thm: uniqueness of measure}
Let $F_p$ be the $p$-harmonic representative of $\rho$, and let
$$\dif u_p := |F_p|^{p - 2} \star F_p.$$
After taking a subsequence, $u_p$ converges in $L^1_\loc(\tilde M)$ to a function $u$ of least gradient such that $\mu := |\dif u| \in \mathcal M(\lambda_\rho)$.
If $M$ is a hyperbolic surface and $\lambda_\rho$ only has closed leaves, then $\mu$ is independent of the choice of subsequence.
\end{theorem}

The proof of Theorem \ref{thm: uniqueness of measure} suggests that if $\lambda_\rho$ has a uniquely ergodic sublamination $\kappa$ which has strictly larger Hausdorff dimension than the rest of $\lambda_\rho$ (and $M$ is not a closed hyperbolic surface), then $\mu$ conjecturally should be the unique measure on $\kappa$, and so should be independent of the choice of subsequence.

Theorems \ref{existence of calibrated lam} and \ref{thm: uniqueness of measure} are explicitly based on theorems about the earthquake norm, and there are also versions of Theorem \ref{vertex counting} and \ref{Auer Bangert thm} for the earthquake norm proven in \cite{huang2024earthquakemetricteichmullerspace}.
In \S\ref{Teichmuller conjectures} we conjecture a version of Corollary \ref{extreme points are indecomposable} for the earthquake norm.

\begin{ack}
This work was closely inspired by ideas in \cite{Auer01} of Franz Auer and Victor Bangert; I am especially grateful to Victor Bangert for allowing me to view their unpublished manuscript \cite{Auer12}.
I also thank Georgios Daskalopoulos and Karen Uhlenbeck for helpful discussions, and James Farre, Yi Huang, Zhenhua Liu, Ben Lowe, and the anonymous referee for helpful comments on an earlier draft.
This research was supported by the National Science Foundation's Graduate Research Fellowship Program under Grant No. DGE-2040433.
\end{ack}

\section{Preliminaries}\label{prevResults}
\subsection{Notation}
Unless otherwise noted, $M$ always denotes a closed oriented Riemannian manifold of dimension $2 \leq d \leq 7$.
The operator $\star$ is the Hodge star on $M$.
We denote the musical isomorphisms by $\sharp, \flat$.
We write $H^\ell$ for de Rham cohomology, but never a Sobolev space, which we instead denote $W^{\ell, p}$.
The second fundamental form of a submanifold $N$ is $\Two_N$.
If $K$ is a closed compact subset of a topological vector space, $\mathcal E(K)$ is the set of extreme points of $K$.

The sheaf of $\ell$-forms is denoted $\Omega^\ell$, and the sheaf of closed $\ell$-forms is denoted $\Omega^\ell_{\rm cl}$.
We assume that $\ell$-forms are $L^1_\loc$, but \emph{not} that they are continuous; hence $\dif$ must be meant in the sense of distributions.

We write $A \lesssim_\theta B$ to mean that $A \leq CB$, where $C > 0$ is a constant that only depends on $\theta$.

\subsection{Differential forms in \texorpdfstring{$L^\infty$}{L-infinity}}
In this section, one can allow $M$ to be an arbitrary complete Riemannian manifold; compactness is unnecessary.

One of the main technical difficulties that we shall have to deal with is that we cannot prove the existence of continuous calibrations in general, and so we shall need to study differential forms which are merely in $L^\infty$.
Such a form $F$ does not need to be well-defined on a set of zero measure, so in general, it does not make sense to integrate $F$ along a submanifold of $M$.

\begin{theorem}[$L^\infty$ Poincar\'e lemma]\label{Hodge theorem}
Let $x \in M$, and $0 \leq k \leq d - 1$.
Then there exists $r_* > 0$ which depends only on $\Riem_M$ near $x$ and the injectivity radius of $x$, such that for every $0 < r \leq r_*$ and $F \in L^\infty(B(x, r), \Omega^{k + 1}_{\rm cl})$, there exists a H\"older continuous $k$-form $A$ such that $F = \dif A$.
\end{theorem}
\begin{proof}
We may choose $r_*$ so that the exponential map $B_{\RR^d}(0, r_*) \to B(x, r_*)$ is a diffeomorphism which induces topological isomorphisms for every function space under consideration.
Thus it is no loss to replace $B(x, r)$ with the unit euclidean ball $\Ball^d$.
By the main theorem of \cite{Costabel2010}, for every $1 < p < \infty$ there is a continuous right inverse to the exterior derivative
\[\begin{tikzcd}
	{W^{1, p}(\Ball^d, \Omega^{\ell - 1})} & {L^p(\Ball^d, \Omega^\ell_{\rm cl})}
	\arrow["{\mathrm d}", from=1-1, to=1-2]
\end{tikzcd}.\]
The result now follows from the Sobolev embedding theorem if we take $p > d$.
\end{proof}

The next result is a rephrasing of \cite[Theorem 1.2]{Anzellotti1983}, and asserts that closed $L^\infty$ $d - 1$-forms can be integrated along Lipschitz hypersurfaces.

\begin{theorem}[normal trace theorem]\label{integration is welldefined}
Let $\iota: N \to M$ be the inclusion of an oriented Lipschitz hypersurface.
Let $\mathcal X$ be the space of $F \in L^\infty(M, \Omega^{d - 1})$ such that the components of $\dif F$ are Radon measures.
Then the pullback $\iota^*$ of $d - 1$-forms extends to a bounded linear map
$$\iota^*: \mathcal X \to L^\infty(N, \Omega^{d - 1})$$
satisfying the estimate
\begin{equation}\label{integral over chain is linfinity}
	\|\iota^* F\|_{L^\infty(N)} \leq \|F\|_{L^\infty(M)}.
\end{equation}
\end{theorem}

The \dfn{comass} of a differential $k$-form $F$ is
$$\|F\|_{L^\infty_*} := \sup_{\Sigma \subset M} \frac{1}{\vol(\Sigma)} \int_\Sigma F,$$
where the supremum ranges over all oriented $k$-dimensional submanifolds $\Sigma$.
It is clear that $\|F\|_{L^\infty_*} \leq \|F\|_{L^\infty}$, but if $F$ is a $d - 1$-form, then the converse holds as well; we shall often use this fact without comment.

A \dfn{$k$-current of finite mass} is a continuous linear functional on the space 
$$C^0(M, \Omega^{d - k}) \cap L^\infty(M, \Omega^{d - k})$$
of bounded continuous $d - k$-forms.\footnote{Be warned: this convention agrees with currents in algebraic geometry, where currents are viewed as generalizations of forms, but not geometric measure theory, where currents are viewed as generalizations of submanifolds.}
We denote the action of a current $T$ on a form $\varphi$ by $\int_M T \wedge \varphi$.

The \dfn{mass} of a $k$-current $T$ is
$$\Mass(T) := \sup_{\|F\|_{L^\infty_*} \leq 1} \int_M T \wedge F.$$
If $T$ represents a $d - k$-dimensional submanifold $\Sigma$, in the sense that $\int_M T \wedge F = \int_\Sigma F$, then $\Mass(T) = \vol(\Sigma)$.
A function $u \in L^1_\loc(M)$ has \dfn{bounded variation}, denoted $u \in BV(M)$, if $\dif u$ is a $1$-current of finite mass, in which case $\Mass(\dif u)$ is the \dfn{total variation} of $u$, and we write $\int_M \star |\dif u|$ to mean $\Mass(\dif u)$.

One cannot multiply two arbitrary distributions, but one can define $\dif u \wedge F$ when $u \in BV$, $F \in L^\infty$, and $\dif F = 0$.
More precisely, we have:

\begin{definition}
Let $u \in BV(M, \Omega^k)$ and $F \in L^\infty(M, \Omega^{d - k - 1})$.
Assume that $\dif F \in L^d(M, \Omega^{d - k})$.
Then the \dfn{Anzellotti wedge product} of $\dif u$ and $F$ is the distribution $\dif u \wedge F$, such that for every test function $\chi \in C^\infty_\cpt(M, \RR)$,
$$\langle \dif u \wedge F, \chi\rangle := -\int_M \chi u \wedge \dif F - \int_M \dif \chi \wedge u \wedge F.$$
\end{definition}

The next theorem is essentially \cite[Theorem 1.5]{Anzellotti1983}, but we sketch the argument because Anzellotti did not formulate it in such generality.

\begin{theorem}[Anzellotti's theorem]\label{Anzellotti wedge product exists}
Let $u \in BV(M, \Omega^k)$, $F \in L^\infty(M, \Omega^{d - k - 1})$, and $\dif F \in L^d(M, \Omega^{d - k})$.
Then the Anzellotti wedge product $\dif u \wedge F$ is well-defined as a distribution.
In fact, $\dif u \wedge F$ is a signed Radon measure, and 
$$\Mass(\dif u \wedge F) \leq \Mass(\dif u) \|F\|_{L^\infty_*}.$$
In particular, if $k = 0$,
\begin{equation}\label{Anzellotti Holder inequality}
\Mass(\dif u \wedge F) \leq \|F\|_{L^\infty} \int_M \star |\dif u|.
\end{equation}
\end{theorem}
\begin{proof}
By the $BV$ Sobolev embedding theorem, \cite[\S5.6]{evans2015measure}, for every $\chi \in C^\infty_\cpt(M)$, $\chi u$ belongs to the dual space of $L^d(M, \Omega^{d - k})$.
Therefore for every $\chi \in C^\infty_\cpt(M)$, $\langle \dif u \wedge F, \chi\rangle$ is finite, so $\dif u \wedge F$ is well-defined as a distribution.

Suppose that $\supp \chi \Subset U$ for some $U \Subset M$.
If $u$ is sufficiently smooth, then an integration by parts gives 
$$|\langle \dif u \wedge F, \chi\rangle| = \left|\int_M \chi \dif u \wedge F\right| \leq \|F\|_{L^\infty_*} \|\chi\|_{C^0} \Mass(1_U \dif u).$$
In general, we can find a sequence $(u_n) \subset C^\infty$ such that $u_n \weakto^* u$ in $BV$.
Then $u_n \weakto u$ in $L^{\frac{d}{d - 1}}$ and $\dif u_n \weakto^* \dif u$ as currents of locally finite mass.
Since we are testing $\dif u$ against the $L^d$ form $\chi F$,
$$|\langle \dif u \wedge F, \chi\rangle| \leq \liminf_{n \to \infty} |\langle \dif u_n \wedge F, \chi\rangle| \leq \|F\|_{L^\infty_*} \|\chi\|_{C^0} \liminf_{n \to \infty} \Mass(1_U \dif u_n).$$
But, by the portmanteau theorem \cite[Theorem 17.20]{kechris2012classical},
$$\liminf_{n \to \infty} \Mass(1_U \dif u_n) \leq \liminf_{n \to \infty} \Mass(1_{\overline U} \dif u_n) \leq \Mass(\dif u)$$
which gives the desired estimate, since we only used the $C^0$ norm of $\chi$.
\end{proof}

\subsection{Calibrated geometry}\label{calibrated review sec}
We recall calibrated geometry, which was developed by Harvey and Lawson \cite{Harvey82}.

\begin{definition}
A \dfn{calibration} is a $k$-form $F$ such that $\dif F = 0$ and $\|F\|_{L^\infty_*} = 1$.
If $\Sigma$ is a $k$-dimensional submanifold, and $F$ pulls back to the Riemannian volume form of $\Sigma$, we say that $\Sigma$ is \dfn{$F$-calibrated}.
\end{definition}

If $\Sigma$ is $F$-calibrated, then for any $k - 1$-dimensional submanifold $\Lambda$,
$$\vol(\Sigma) = \int_\Sigma F = \int_{\Sigma + \partial \Lambda} F \leq \vol(\Sigma + \partial \Lambda),$$
so that $\Sigma$ is area-minimizing.
On the other hand, if $A \in W^{1, \infty}(M, \Omega^{k - 1})$, and $\Sigma$ is a closed $F$-calibrated submanifold, then
$$\|F\|_{L^\infty_*} = 1 = \frac{1}{\vol(\Sigma)} \int_\Sigma F = \frac{1}{\vol(\Sigma)} \int_\Sigma F + \dif A \leq \|F + \dif A\|_{L^\infty_*},$$
so $F$ minimizes its comass in its cohomology class if it calibrates a closed hypersurface.

The definition of $F$-calibration extends to currents.
If $F$ is a calibration $k$-form, a $d - k$-current $T$ is \dfn{$F$-calibrated} if
$$\int_M T \wedge F = \Mass(T).$$
By Anzellotti's theorem, Theorem \ref{Anzellotti wedge product exists}, this definition makes sense as long as $T$ has locally finite mass.
If $T$ is $F$-calibrated, then for any $d - k - 1$-current $S$, $\Mass(T) \leq \Mass(T + \dif S)$.

The comass and mass induce norms on cohomology and homology.
The \dfn{stable norm} $\|\cdot\|_1$ on $H_k(M, \RR)$ is defined by 
$$\|\theta\|_1 := \inf_{\PD([T]) = \theta} \Mass(T),$$
where $\PD(\omega)$ is the Poincar\'e dual of the class $\omega$, and $T$ ranges over $d - k$-currents.
The \dfn{costable norm} $\|\cdot\|_\infty$ is the dual norm of $\|\cdot\|_1$ on $H^k(M, \RR)$.
The following theorem is a special case of the main theorem of \cite[\S4]{Federer1974} but it is essential to us, so we sketch the proof.

\begin{theorem}
For every $\rho \in H^k(M, \RR)$,
$$\|\rho\|_\infty = \min_{[F] = \rho} \|F\|_{L^\infty_*},$$
where $F$ ranges over closed measurable $k$-forms of class $\rho$.
\end{theorem}
\begin{proof}
For each representative $F$ of $\rho$, and with $T$ ranging over all $d - k$-currents,
$$\|\rho\|_\infty = \sup_{\substack{\Mass(T) \leq 1 \\ \dif T = 0}} \int_M T \wedge F,$$
so in particular $\|\rho\|_\infty \leq \|F\|_{L^\infty_*}$.
Conversely, for each $\kappa \in L^1(M, \Omega^{d - k})$ such that $\dif \kappa = 0$, let 
$$\Psi(\kappa) := \langle \rho, \PD([\kappa])\rangle.$$
Then 
$$|\Psi(\kappa)| \leq \|\rho\|_\infty \|\PD([\kappa])\|_1 \leq \|\rho\|_\infty \Mass(\kappa)$$
so by the Hanh-Banach theorem, there exists $F \in L^\infty(M, \Omega^k)$ such that $\|F\|_{L^\infty_*} \leq \|\rho\|_\infty$ and for every $\kappa \in L^1(M, \Omega^{d - k})$ such that $\dif \kappa = 0$,
$$\int_M \kappa \wedge F = \langle \rho, \PD([\kappa])\rangle.$$
This implies that $\dif F = 0$ and $[F] = \rho$.
\end{proof}

\subsection{Laminations}\label{sec: lamination review}
We use roughly the same formalism for laminations as in \cite{Morgan88}, which we also used in \cite{BackusCML}.
Let $M$ be a Riemannian manifold.
Fix an interval $I \subset \RR$ and a box $J \subset \RR^{d - 1}$.
A (codimension-$1$, Lipschitz) \dfn{laminar flow box} is a Lipschitz coordinate chart $\Psi: I \times J \to M$ and a compact set $K \subseteq I$, called the \dfn{local leaf space}, such that for each $k \in K$, $\Psi|_{\{k\} \times J}$ is a $C^1$ embedding, and the \dfn{leaf} $\Psi(\{k\} \times J)$ is a $C^1$ complete hypersurface in $\Psi(I \times J)$.
Two laminar flow boxes belong to the same \dfn{laminar atlas} if the transition map preserves the local leaf spaces.

\begin{definition}
A (codimension-$1$, Lipschitz) \dfn{lamination} $\lambda$ is a closed nonempty set $\supp \lambda$ and a maximal laminar atlas $\{(\Psi_\alpha, K_\alpha): \alpha \in A\}$ such that
$$\supp \lambda \cap \Psi_\alpha(I \times J) = \Psi_\alpha(K_\alpha \times J).$$
\end{definition}

Note carefully that the leaves of a lamination will typically not be embedded, but merely injectively immersed. 
The following theorem allows us to construct laminations without explicitly constructing their flow boxes, provided that the leaves are minimal hypersurfaces.

\begin{theorem}[{\cite[Theorem A]{BackusCML}}]\label{disjoint surfaces are lamination}
Let $\mathcal S$ be a set of disjoint minimal hypersurfaces in $M$, such that $\bigcup_{N \in \mathcal S} N$ is a closed set, and $\sup_{N \in \mathcal S} \|\Two_N\|_{C^0} < \infty$.
Then $\mathcal S$ is the set of leaves of a Lipschitz lamination $\lambda$, and the normal vector to the leaves of $\lambda$ extends to a Lipschitz section of a line bundle over $M$.
\end{theorem}

In our application, $d \leq 7$ and the hypersurfaces in $\mathcal S$ are stable, so we can check the hypothesis on curvature in Theorem \ref{disjoint surfaces are lamination} using \cite{Schoen81}.
In my experience, it is a common misconception that the hypothesis on curvature can be removed, but the next example shows that it cannot be.

\begin{example}\label{ex: stack of catenoids}
Let $M$ be the unit ball of $\RR^3$, let
\begin{align*}
\iota_n: \RR \times \Sph^1 &\to \RR^3 \\
(z, \theta) &\mapsto (2^{-n} \cosh(2^n z) \cos \theta, 2^{-n} \cosh(2^n z) \sin \theta, z + 1/n)
\end{align*}
and let $N_n := M \cap (\iota_n)_*(\RR \times \Sph^1)$.
So if $\iota_n(z, \theta) \in N_n$ then $2^{-n} \cosh(2^n z) \leq 1$.
In this case, if $n \geq 3$ then
$$2^n z \leq \arcosh(2^n) \leq n,$$
and in particular
$$N_n \subset \{(x, y, z) \in M: 1/n - n/2^n \leq z \leq 1/n + n/2^n\}.$$
It follows that if $n, m \geq 12$ then $N_n \cap N_m = \emptyset$.
Let $\mathcal S$ be the set of all $N_n$s where $n \geq 12$, and the $z$-axis; one could call such a structure a \dfn{stack of catenoids}.
A stack of catenoids is not a lamination, even though its leaves are disjoint and minimal, and have closed union.
The leaves of a stack of catenoids have Morse index $\leq 1$ and area $\leq 5\pi$, which is ``almost as good'' as having bounded Gaussian curvature, but even this is not enough.
\end{example}

An arbitrary lamination cannot be viewed as a current, but following Ruelle and Sullivan \cite{Ruelle75}, we view laminations which have been equipped with transverse measures and orientations as currents, so our next task is to define the Ruelle-Sullivan current.

\begin{definition}
Let $\lambda$ be a lamination with laminar atlas $\{(\Psi_\alpha, K_\alpha): \alpha \in A\}$. Then:
\begin{enumerate}
\item $\lambda$ is equipped with an \dfn{orientation} if the transition maps $\Psi_\alpha^{-1} \circ \Psi_\beta$ are orientation-preserving.
\item A \dfn{transverse measure} $\mu$ to $\lambda$ consists of Radon measures $\mu_\alpha$ on each local leaf space $K_\alpha$, such that the transition maps $\Psi_\alpha^{-1} \circ \Psi_\beta$ send $\mu_\beta$ to $\mu_\alpha$, and $\supp \mu_\alpha = K_\alpha$.
The pair $(\lambda, \mu)$ is a \dfn{measured lamination}.
\item Suppose that $\lambda$ is oriented, $\mu$ is a transverse measure to $\lambda$, and $\{\chi_\alpha: \alpha \in A\}$ is a partition of unity subordinate to $\{\Psi_\alpha(I \times J): \alpha \in A\}$. The \dfn{Ruelle-Sullivan current} $T_\mu$ acts on $\varphi \in C^0_\cpt(M, \Omega^{d - 1})$ by 
$$\int_M T_\mu \wedge \varphi := \sum_{\alpha \in A} \int_{K_\alpha} \left[\int_{\{k\} \times J} (\Psi_\alpha^{-1})^* (\chi_\alpha \varphi)\right] \dif \mu_\alpha(k).$$
\end{enumerate}
\end{definition}

It bears repeating that in our convention, a lamination $\lambda$ is nonempty, and if $\mu$ is a transverse measure to $\lambda$, then $\supp \mu = \supp \lambda$.

Let $(\lambda, \mu)$ be a measured oriented lamination.
It is a straightforward modification of the arguments of \cite[\S8]{daskalopoulos2020transverse} to show that the Ruelle-Sullivan current $T_\mu$ is a closed $1$-current which is well-defined, in the sense that $T_\mu$ does not depend on the choice of laminar atlas.
Furthermore, by \cite[Lemma 3.1]{BackusCML},
\begin{equation}\label{decomposition of Ruelle Sullivan}
T_\mu = \normal_\lambda^\flat \mu
\end{equation}
where $\normal_\lambda^\flat$ is the conormal $1$-form to $\lambda$ and $\mu(U) := \int_U \star |T_\mu|$ for every open set $U$.
Often we leave $\mu$ implicit and just write $T_\lambda$ for $T_\mu$.

Let $\lambda$ be a lamination.
A Borel set $E \subseteq \supp \lambda$ is \dfn{saturated} if, for every leaf $N$ of $\lambda$ such that $N \cap E$ is nonempty, $N \subseteq E$.
Every leaf of $\lambda$ is Borel, and therefore saturated.
A \dfn{sublamination} of $\lambda$ is a closed saturated set.
Every sublamination of $\lambda$ inherits the flow boxes of $\lambda$ and therefore is itself a lamination.

\begin{lemma}\label{existence of intersections}
Let $\mathscr S$ be a nonempty set of laminations.
Suppose that there exists a hypersurface which is a leaf of every lamination in $\mathscr S$.
Then there exists a lamination whose set of leaves is the intersection of the sets of leaves of the laminations in $\mathscr S$.
\end{lemma}
\begin{proof}
Let $\lambda \in \mathscr S$, and let $(\Psi_\alpha, K_\alpha)_{\alpha \in A}$ be a laminar atlas for $\lambda$.
Let $K'_\alpha$ be the set of $k \in K_\alpha$ such that for every $\kappa \in \mathscr S$, there exists a leaf $N$ of $\kappa$ such that
$$(\Psi_\alpha)_*(\{k\} \times J) \subseteq N.$$
It is clear that this property is preserved by transition maps.
Then $K_\alpha'$ is an intersection of compact sets (since the local leaf space of each $\kappa \in \mathscr S$ is compact), so $K_\alpha'$ is compact.
The hypersurface which is a common leaf of every lamination in $\mathscr S$ witnesses that for some $\alpha$, $K_\alpha'$ is nonempty.
Therefore $(\Psi_\alpha, K'_\alpha)_{\alpha \in A}$ is a laminar atlas.
The fact that $K_\alpha'$ is compact for every $\alpha$ implies that the supposed lamination whose atlas is $(\Psi_\alpha, K'_\alpha)_{\alpha \in A}$ has a closed support.
\end{proof}

We shall also need a form of the \dfn{Morgan--Shalen decomposition}, \cite[Theorem I.3.2]{Morgan88}, of a measured lamination.
To formulate it, let us say that a lamination $\lambda$ is \dfn{exceptional}\footnote{Exceptional laminations are often called \dfn{minimal}, but that clashes with the use of the word ``minimal'' to refer to vanishing mean curvature, so we have not adopted this terminology.} if every leaf of $\lambda$ is dense in $\supp \lambda$, and $\lambda$ is not a single closed leaf.
A lamination $\lambda$ is a \dfn{parallel family of closed leaves} if there exists a closed leaf $N$ of $\lambda$ with trivial normal bundle, such that every leaf of $\lambda$ is a section of the normal bundle of $N$.

\begin{theorem}[Morgan--Shalen decomposition]\label{MorganShelan}
Suppose that $M$ is closed and oriented.
For every measured oriented lamination $\lambda$, one of the following holds:
\begin{enumerate}
\item $\lambda$ is a foliation with a dense leaf.
\item $\lambda$ is the disjoint union of finitely many clopen sublaminations $\kappa$, such that either $\kappa$ is exceptional, or $\kappa$ is a parallel family of closed leaves.
\end{enumerate}
\end{theorem}
\begin{proof}
First observe that the proof of \cite[Theorem I.3.2]{Morgan88} goes through for any lamination $\lambda$ such that no leaf of $\lambda$ is dense in $M$, even if $\lambda$ is a foliation.
It then remains to rule out the case that $\kappa$ is a family of sections of a nontrivial normal bundle of a closed leaf: this holds because $\lambda$ is oriented.
\end{proof}

\subsection{Perimeter-minimizing sets}
In this section only, $M$ denotes a complete Riemannian manifold of bounded curvature; we do not assume that $d = 7$ or $M$ is closed.
A Borel set $U \subseteq M$ is \dfn{perimeter-minimizing} if for every ball $B \subseteq M$ and every Borel set $V \subseteq M$ such that $1_U = 1_V$ on $\partial B$, $\vol(\partial U) \leq \vol(\partial V)$ in the sense of currents.
In the proof of Theorem \ref{Auer Bangert thm}(\ref{ABt2}), we shall need an estimate on perimeter-minimizing sets, which we now prove; the reader can ignore this section until then.
See \cite[Chapter 5]{Giusti77} for the proof when $M$ is an open subset of euclidean space.

\begin{lemma}
There are constants $\delta, c > 0$ which only depend on $d$ such that for every $r \in (0, \delta \|\Riem_M\|_{C^0}^{-1/2}]$ and $x \in M$ such that $\dist(x, \partial M) > r$,
\begin{equation}\label{reverse isoperimetric inequality}
\vol(U \cap B(x, r)) \geq cr^d.
\end{equation}
\end{lemma}
\begin{proof}
If we take $\delta$ small enough, then we can approximate $B(x, r)$ by a euclidean ball so well that, by the euclidean isoperimetric inequality, for every $0 < \rho \leq r$,
$$\vol(\partial(U \cap B(x, \rho))) \geq \frac{1}{2c_d} \vol(U \cap B(x, \rho))^{\frac{d - 1}{d}},$$
where $c_d > 0$ is the euclidean isoperimetric constant.
We can reason as in the proof of \cite[Proposition 5.14]{Giusti77} to see that for almost every $0 < \rho < r$,
$$\frac{\dif}{\dif \rho} \vol(U \cap B(x, \rho)) \geq \frac{1}{2} \vol(\partial(U \cap B(x, \rho))) \geq \frac{1}{4c_d} \vol(U \cap B(x, \rho))^{\frac{d - 1}{d}}.$$
Let $f(\rho) := \vol(U \cap B(x, \rho))^{1/d}$, so that 
$$f'(\rho) = \frac{\vol(U \cap B(x, \rho)^{\frac{d - 1}{d}})}{d} \frac{\dif}{\dif \rho} \vol(U \cap B(x, \rho)) \geq \frac{1}{4dc_d}.$$
Therefore $f(\rho) \geq \rho/(4dc_d)$, as desired.
\end{proof}

\begin{lemma}\label{unbounded implies infinite measure}
For every unbounded perimeter-minimizing set $U \subseteq M$, $\vol(U) = \infty$.
\end{lemma}
\begin{proof}
Let $r := \min(1, \delta \|\Riem_M\|_{C^0}^{-1/2})$ where $\delta$ is as in the previous lemma.
Since $U$ is unbounded, there is an infinite $2r$-separated set $S \subset U$.
Then the set $U \cap \bigcup_{x \in S} B(x, r)$ has infinite volume by (\ref{reverse isoperimetric inequality}).
\end{proof}

\section{Calibrated laminations and functions of least gradient}
\subsection{Calibrated laminations}
Let $M$ be a closed oriented Riemannian manifold.
Let $F$ be a calibration on $M$, and $\lambda$ a measured oriented lamination in $M$.
There are two things that one could conceivably mean by saying that $\lambda$ is $F$-calibrated: that every leaf of $\lambda$ is $F$-calibrated, or that the Ruelle-Sullivan current, $T_\lambda$, is $F$-calibrated.
The purpose of this section is to show that these two notions are equivalent.

\begin{definition}
Let $F \in L^\infty(M, \Omega^{d - 1})$ be a calibration.
A lamination $\lambda$ is \dfn{$F$-calibrated} if every leaf of $\lambda$ is $F$-calibrated.
\end{definition}

By the normal trace theorem, Theorem \ref{integration is welldefined}, this definition makes sense.
Of course, one is only really interested in calibrated laminations if they are mass-minimizing, so now we recall that the \dfn{mass} of a measured oriented lamination $\lambda$ is
$$\Mass(\lambda) := \Mass(T_\lambda).$$
Since a current can be approximated by smooth $1$-forms in the weakstar topology on currents, every current has a cohomology class $[T] \in H^1(M, \RR)$.
Thus, the \dfn{homology class} $[\lambda] \in H_{d - 1}(M, \RR)$ is the Poincar\'e dual of $[T_\lambda]$.

\begin{definition}
Let $\lambda$ be a measured oriented lamination, and assume that $M$ is compact.
Then $\lambda$ is \dfn{homologically minimizing}, if for every measured oriented lamination $\kappa$ such that $[\lambda] = [\kappa]$,
$$\Mass(\lambda) \leq \Mass(\kappa).$$
\end{definition}

Let $(\lambda, \mu)$ be a measured oriented lamination.
Let $(\chi_\alpha)$ be a locally finite partition of unity subordinate to a laminar atlas $(U_\alpha, K_\alpha)$ for $\lambda$.
If $\sigma_{\alpha, k}$ denotes the leaf in $U_\alpha$ corresponding to the real number $k \in K_\alpha$, then the definition of the Ruelle-Sullivan current unpacks as
\begin{equation}\label{coarea formula on laminations}
\int_M T_\lambda \wedge F = \sum_\alpha \int_{K_\alpha} \int_{\sigma_{\alpha, k}} \chi_\alpha F \dif \mu_\alpha(k).
\end{equation}
Since $T_\lambda$ and $F$ are closed, if $M$ is closed, then the left-hand side of (\ref{coarea formula on laminations}) is a homological invariant:
\begin{equation}\label{Ruelle Sullivan homology}
\int_M T_\lambda \wedge F = \langle [F], [\lambda]\rangle.
\end{equation}

\begin{lemma}\label{calibration condition}
Let $F$ be a calibration.
Let $T_\lambda$ be the Ruelle-Sullivan current of a measured oriented lamination $\lambda$.
Then the following are equivalent:
\begin{enumerate}
\item $T_\lambda$ is $F$-calibrated.
\item $\lambda$ is $F$-calibrated.
\end{enumerate}
\end{lemma}
\begin{proof}
First suppose that $T_\lambda$ is $F$-calibrated.
Let $(\chi_\alpha)$ be a locally finite partition of unity subordinate to an open cover $(U_\alpha)$ of flow boxes for $\lambda$, let $(K_\alpha)$ be the local leaf spaces, and let $(\mu_\alpha)$ be the transverse measure.
After refining $(U_\alpha)$ we may assume that $U_\alpha$ is a ball which satisfies the hypotheses of the $L^\infty$ Poincar\'e lemma, Theorem \ref{Hodge theorem}. After shrinking $U_\alpha$ we may assume that $\chi_\alpha > 0$ on $U_\alpha$.
Then for leaves $\sigma_{\alpha,k}$, we rewrite (\ref{coarea formula on laminations}) as 
$$\Mass(\lambda) = \int_M T_\lambda \wedge F = \sum_\alpha \int_{K_\alpha} \int_{\sigma_{\alpha,k}} \chi_\alpha F \dif \mu_\alpha(k).$$
Let $\dif S_{\alpha,k}$ be the surface measure on $\sigma_{\alpha,k}$.
Then
$$\int_M \chi_\alpha \star |T_\lambda| = \int_{K_\alpha} \int_{\sigma_{\alpha,k}} \chi_\alpha \dif S_{\alpha,k} \dif \mu_\alpha(k),$$
so summing in $\alpha$, we obtain 
\begin{equation}\label{calibration condition contr}
\sum_\alpha \int_{K_\alpha} \int_{\sigma_{\alpha,k}} \chi_\alpha F \dif \mu_\alpha(k) = \Mass(\lambda) = \sum_\alpha \int_{K_\alpha} \int_{\sigma_{\alpha,k}} \chi_\alpha \dif S_{\alpha,k} \dif \mu_\alpha(k).
\end{equation}

We claim that $\lambda$ is \dfn{almost calibrated} in the sense that for every $\alpha$ and $\mu_\alpha$-almost every $k$, $\sigma_{\alpha, k}$ is calibrated.
If this is not true, then we may select $\beta$ and $K \subseteq K_\beta$ with $\mu_\beta(K) > 0$, such that for every $k \in K$, $\int_{\sigma_{\beta, k}} F < \vol(\sigma_{\beta, k})$.
Since $0 < \chi_\beta \leq 1$ and $F/\dif S_{\beta, k} \leq 1$ on $\sigma_{\beta, k}$, this is only possible if 
$$\int_{\sigma_{\beta, k}} \chi_\beta F < \int_{\sigma_{\beta, k}} \chi_\beta \dif S_{\beta, k}.$$
Integrating over $K$, and using the fact that in general we have $\int_{\sigma_{\alpha, k}} \chi_\alpha F \leq \int_{\sigma_{\alpha, k}} \chi_\alpha \dif S_{\alpha, k}$, we conclude that 
$$\sum_\alpha \int_{K_\alpha} \int_{\sigma_{\alpha, k}} \chi_\alpha F \dif \mu_\alpha(k) < \sum_\alpha \int_{K_\alpha} \int_{\sigma_{\alpha, k}} \chi_\alpha \dif S_{\alpha, k} \dif \mu_\alpha(k)$$
which contradicts (\ref{calibration condition contr}).

To upgrade $\lambda$ from an almost calibrated lamination to a calibrated lamination, we first, given $\sigma_{\alpha, k}$, choose $k_j$ such that $\sigma_{\alpha, k_j}$ is calibrated and $k_j \to k$.
By Theorem \ref{Hodge theorem}, we can find a continuous $d - 2$-form $A$ defined near $\sigma_{\alpha, k}$ with $F = \dif A$.
This justifies the following application of Stokes' theorem: 
$$\int_{\sigma_{\alpha, k}} F = \int_{\partial \sigma_{\alpha, k}} A.$$
Since $k_j \to k$, and $A$ is continuous,
\begin{align*}
\vol(\sigma_{\alpha, k}) &= \lim_{j \to \infty} \vol(\sigma_{\alpha, k_j}) = \lim_{j \to \infty} \int_{\sigma_{\alpha, k_j}} F = \lim_{j \to \infty} \int_{\partial \sigma_{\alpha, k_j}} A = \int_{\partial \sigma_{\alpha, k}} A = \int_{\sigma_{\alpha, k}} F.
\end{align*}

To establish the converse, suppose that $\lambda$ is $F$-calibrated, and let notation be as above.
Since $\lambda$ is $F$-calibrated, for every $\alpha$ and every $k$, the area form on $\sigma_{\alpha, k}$ is $F$. Therefore
\begin{align*}
\int_M T_\lambda \wedge F &= \sum_\alpha \int_{K_\alpha} \int_{\sigma_{\alpha, k}} \chi_\alpha F \dif \mu_\alpha(k) = \Mass(T_\lambda). \qedhere
\end{align*}
\end{proof}

\begin{lemma}\label{properties of calibrated laminations}
Suppose that $M$ is closed.
Let $F$ be a calibration, and let $\lambda$ be a measured oriented $F$-calibrated lamination.
Then:
\begin{enumerate}
\item $\lambda$ is homologically minimizing.
\item If $G$ is a calibration and cohomologous to $F$, then $\lambda$ is $G$-calibrated.
\end{enumerate}
\end{lemma}
\begin{proof}
Every leaf of $\lambda$ is $F$-calibrated, hence minimal.
Since $\lambda$ is $F$-calibrated, so is $T_\lambda$ by Lemma \ref{calibration condition}, but then by (\ref{Ruelle Sullivan homology}), it follows that $T_\lambda$ is $G$-calibrated, and hence $\lambda$ is $G$-calibrated.
Moreover, since $T_\lambda$ is $F$-calibrated, a calibration argument shows that $\lambda$ is homologically minimizing.
\end{proof}

\subsection{Functions of least gradient}
The natural ``dual objects'' to calibrations are functions of least gradient, which we now define.

We begin with some topological preliminaries.
Let $M$ be a closed oriented Riemannian manifold of dimension $d$, and let $\tilde M \to M$ be the universal covering map.
Any homomorphism 
$$\alpha: \pi_1(M) \to \RR$$
induces a homomorphism $\alpha: H_1(M, \RR) \to \RR$.
Thus $\alpha$ is an element of $H^1(M, \RR)$ and by Poincar\'e duality, we view it as an element of $H_{d - 1}(M, \RR)$.
Concretely, the following are equivalent for a function $u \in BV_\loc(\tilde M, \RR)$:
\begin{enumerate}
\item $u$ is \dfn{$\alpha$-equivariant}, meaning that for every deck transformation $c \in \pi_1(M)$, and every $x \in \tilde M$,
$$u(cx) = u(x) + \alpha(c).$$
\item $\dif u$ descends to a $1$-current on $M$ whose cohomology class is the Poincar\'e dual of $\alpha$.
\end{enumerate}
In either case we write $[\dif u] = \alpha$, and write $\int_M \star |\dif u|$ or $\Mass(\dif u)$ to refer to the mass of the $1$-current that $\dif u$ induces on $M$.
If $[\dif u] = 0$ then we identify $u$ with the function that it induces on $M$.

\begin{definition}
Let $u \in BV(\tilde M, \RR)$ be a $\pi_1(M)$-equivariant function.
Suppose that, for every $v \in BV(M, \RR)$,
$$\int_M \star |\dif u| \leq \int_M \star |\dif u + \dif v|.$$
Then $u$ has \dfn{least gradient}.
\end{definition}

An $\alpha$-equivariant function $u$ has least gradient iff $\int_M \star |\dif u|$ is the stable norm, $\|\alpha\|_1$, of $\alpha$, which we defined in \S\ref{calibrated review sec}.

\begin{lemma}\label{existence for least gradient}
For each $\alpha \in H_{d - 1}(M, \RR)$, there exists an $\alpha$-equivariant function of least gradient on $\tilde M$.
\end{lemma}
\begin{proof}
The argument here is a standard application of the direct method of the calculus of variations, so we just sketch the proof.
Let $(u_n)$ be a sequence of $\alpha$-equivariant functions such that
$$\lim_{n \to \infty} \int_M \star |\dif u_n| = \|\alpha\|_1.$$
This sequence is bounded in $BV_\loc(\tilde M, \RR)$, so by Alaoglu's theorem, it has a subsequence which converges in the weakstar topology of $BV_\loc$ to some function $u$ such that $\Mass(\dif u) \leq \|\alpha\|_1$.
By testing $\dif u_n$ against smooth $d - 1$-forms on $M$, we see that $[\dif u] = \alpha$ and so $u$ has least gradient.
\end{proof}

\begin{theorem}\label{least gradient iff lamination}
Assume that $d \leq 7$.
Let $u \in BV(\tilde M, \RR)$ be a $\pi_1(M)$-equivariant function which is nonconstant.
The following are equivalent:
\begin{enumerate}
\item $u$ has least gradient.
\item There is a homologically minimizing lamination $\lambda_u$ on $M$ such that:
\begin{enumerate}
\item $T_{\lambda_u} = \dif u$.
\item Every leaf of $\lambda_u$ is a minimal hypersurface.
\item Every leaf of $\lambda_u$ pulls back to a union of subsets of $\tilde M$ of the form $\partial \{u > y\}$ or $\partial \{u < y\}$ for some $y \in \RR$.
\end{enumerate}
\end{enumerate}
\end{theorem}
\begin{proof}
If $u$ has least gradient, then \cite[Theorem B]{BackusCML} implies that there is a measured oriented lamination $\tilde \lambda_u$ of minimal hypersurfaces on $\tilde M$ whose leaves are level sets of $u$, and whose Ruelle-Sullivan current is $\dif u$.
Since $u$ is equivariant, $\tilde \lambda_u$ descends to a lamination $\lambda_u$ on $M$ such that $\Mass(\lambda_u) = \Mass(\dif u)$.
Since $u$ has least gradient, $\lambda_u$ is homologically minimizing.

Conversely, if such a lamination exists, \cite[Theorem B]{BackusCML} implies that $u$ locally has least gradient and $\Mass(\dif u) = \Mass(\lambda_u)$, so $u$ has least gradient.
\end{proof}

Combining the above two results, we see that if $d \leq 7$, every nonzero class in $H_{d - 1}(M, \RR)$ contains a homologically minimizing lamination.

\subsection{Duality of calibrations and laminations}
Recall from \S\ref{calibrated review sec} the definition of the costable norm, $\|\cdot\|_\infty$.
If $F$ is a calibration in a cohomology class $\rho$, then either $\|\rho\|_\infty = 1$ (because $F$ minimizes its $L^\infty$ norm in $\rho$, and $\|F\|_{L^\infty} = 1$), or $F$ calibrates no currents whatsoever.
Conversely, if $\|\rho\|_\infty = 1$, then by Alaoglu's theorem, there is a calibration in $\rho$.

It is natural to ask if there is a \emph{continuous} calibration in $\rho$, as was assumed in \cite{bangert_cui_2017,Freedman_2016}.
If $d = 2$ one might try to generalize the argument of \cite{Evans08} to obtain a H\"older continuous calibration, but if $d \geq 8$ then continuous calibrations need not exist \cite{liu2023homologically}.
The situation that $3 \leq d \leq 7$ remains unclear.
If $\Ric_M \geq 0$, then the Bochner argument shows that the harmonic representative of $\rho$ is a calibration; however, the Bochner argument actually shows that $M = \Sph^1 \times N$ where $N$ is the calibrated hypersurface, so this is not very interesting.

In the setting of the Dirichlet problem for a domain on euclidean space, Maz\'on, Rossi, and Segura de Le\'on \cite{Mazon14} proved that a $BV$ function has least gradient iff it is calibrated by some calibration.
In fact, the same duality holds here, but in the equivariant setting the proof is trivial.

\begin{lemma}\label{Mazon Rossi Segura}
Let $u \in BV_\loc(\tilde M, \RR)$ be an equivariant function.
The following are equivalent: 
\begin{enumerate}
\item $u$ has least gradient.
\item There exists a calibration $F$ on $M$ such that $\dif u$ is $F$-calibrated.
\end{enumerate}
\end{lemma}
\begin{proof}
If $\dif u$ is $F$-calibrated, then we have by Stokes' theorem and (\ref{Anzellotti Holder inequality}) that for any $v \in BV(M, \RR)$,
$$\int_M \star |\dif u| = \int_M \dif u \wedge F = \int_M (\dif u + \dif v) \wedge F \leq \int_M \star |\dif u + \dif v|,$$
so $u$ has least gradient.

Conversely, if $u$ has least gradient, then let $\alpha := [\dif u]$ and choose $\rho \in H^{d - 1}(M, \RR)$ such that $\langle \rho, \alpha\rangle = \|\alpha\|_1$ and $\|\rho\|_\infty = 1$.
In particular, there exists a calibration $F$ such that $[F] = \rho$, and 
$$\int_M \dif u \wedge F = \langle \rho, \alpha\rangle = \|\alpha\|_1 = \int_M \star |\dif u|,$$
so that $u$ has least gradient.
\end{proof}

The above proof motivates the introduction of the following terminology from convex geometry.
A \dfn{flat} in the stable unit sphere $\partial B$ is the intersection of $\partial B$ with a hyperplane.
In particular, every flat is convex.
If $\|\rho\|_\infty = 1$, its \dfn{dual flat} is
$$\rho^* := \{\alpha \in \partial B: \langle \rho, \alpha\rangle = 1\}.$$
This set is convex, compact, and nonempty; in general it does not have to be a singleton.
Every hyperplane in $H_{d - 1}(M, \RR)$ takes the form $\{\alpha \in H_{d - 1}(M, \RR): \langle \rho, \alpha\rangle = t\}$ for some $\rho$ in the costable unit sphere and some $t \in \RR$, so every flat in $\partial B$ is contained in $\rho^*$ for some $\rho \in \partial B^*$.

The next lemma was observed by Bangert and Cui, \cite{bangert_cui_2017}, in the setting that $F$ is continuous and we require no regularity on the laminations involved.

\begin{lemma}\label{calibrated means measured stretch}
Suppose that $M$ is a closed Riemannian manifold of dimension $d \leq 7$.
Let $\rho \in H^{d - 1}(M, \RR)$ satisfy $\|\rho\|_\infty = 1$, and let $F$ be a calibration in $\rho$.
Then there exists an $F$-calibrated measured oriented lamination.
\end{lemma}
\begin{proof}
Choose $\alpha \in \rho^*$, and let $u$ be an $\alpha$-equivariant function of least gradient.
Then $\dif u$ is $F$-calibrated, so the measured oriented lamination $\kappa$ given by Theorem \ref{least gradient iff lamination} is $F$-calibrated by Lemma \ref{calibration condition}.
\end{proof}

In view of the Maz\'on--Rossi--Segura de Le\'on theorem and Lemma \ref{calibrated means measured stretch}, it is natural to conjecture that \emph{if $F$ minimizes its $L^\infty$ norm subject to a boundary condition on a domain $U$ in euclidean space and $\|F\|_{L^\infty} = 1$, then $F$ calibrates some function on $U$}.
The following example shows that this conjecture is false.

\begin{example}\label{boundaries bad}
Let
$$v(x + iy) := \arctan\left(\frac{y}{x}\right)$$
defined on the open disk $U \subset \CC$ bounded by the circle $(x - 2)^2 + y^2 = 1$.
Working in the polar coordinates $x + iy = re^{i\theta}$, one has $\dif v = \dif \theta$, so $|\dif v| = r^{-1}$, which attains its maximum only at the boundary point $x + iy = 1$, where it is $1$.
Therefore $\dif v$ is a calibration.

Let $\iota: \partial U \to \overline U$ be the inclusion map.
If $F$ is another calibration which satisfies the boundary condition $\iota^* F = \iota^* \dif v$, then since $\dif F = 0$, we can write $F = \dif w$.
Then the trace of $w - v$ along $\partial U$ is a constant, which we may assume to be $0$.
But $v$ is \dfn{$\infty$-harmonic}, meaning that
$$\langle \nabla^2 v, \nabla v \otimes \nabla v\rangle = 0.$$
Indeed, in polar coordinates, the euclidean metric becomes
$$g = \dif r^2 + r^2 \dif \theta^2,$$
so the Christoffel symbol ${\Gamma^\theta}_{\theta \theta}$ vanishes.
Then we compute 
$$\langle \nabla \dif \theta, \dif \theta \otimes \dif \theta\rangle = \langle \nabla \dif \theta, \partial_\theta \otimes \partial_\theta \rangle r^{-4} = r^{-4} {\Gamma^\theta}_{\theta \theta} = 0.$$
Since $v$ is $\infty$-harmonic, $\|F\|_{L^\infty} \geq \|\dif v\|_{L^\infty}$ \cite{Crandall2008}, so that the ``costable norm'' of the boundary data $\iota^* \dif v$ is $1$.

But if $u$ is a function on $U$ such that $\dif u$ is $\dif v$-calibrated, then
$$\supp \dif u \cap U \subseteq \{|\dif v| = 1\} \cap U = \emptyset,$$ 
so $u$ is constant!
A more geometric way to visualize this phenomenon is to notice that the streamlines of $v$ -- that is, the integral curves of the gradient of $v$ -- are the circles centered on $0$.
If $u$ was $\dif v$-calibrated and $\gamma$ was a level curve of $u$, then $\gamma$ would be a streamline of $v$, but would also be $\dif v$-calibrated.
We would conclude that circles are geodesics, which is a contradiction.
\end{example}

\section{Construction of the canonical lamination}
\label{canonical sec}
Throughout this section, we fix a closed oriented Riemannian manifold $M$ of dimension $2 \leq d \leq 7$, and a cohomology class $\rho \in H^{d - 1}(M, \RR)$ in the costable unit sphere: $\|\rho\|_\infty = 1$.
We prove Theorem \ref{existence of calibrated lam}: \emph{the set of complete immersed hypersurfaces, which are calibrated by every calibration in $\rho$, is the set of leaves of a lamination with Lipschitz regularity}.

Let $F$ be a calibration in $\rho$.
The set $S := \{|F| = 1\}$ need not be the support of a lamination $\lambda$; and even if it was, we would not be able to conclude that $F$ calibrates $\lambda$.
For example, if $d \geq 3$, then one can exploit the possible nonintegrability of $\star F$ to produce counterexamples \cite[\S4]{bangert_cui_2017}.
More starkly, if $d = 2$, then the main theorem of \cite{daskalopoulos2020transverse} then implies that $S$ contains a geodesic lamination $\lambda$; on the other hand, the main theorem of \cite{backus2024lipschitz} implies that any closed set containing $\supp \lambda$ can be realized as the set $\{|G| = 1\}$ for some calibration $G$ in $\rho$.
If $M$ is hyperbolic, then $\lambda$ has Hausdorff dimension $1$ \cite{Birman1985}, so ``almost every'' closed subset of $M$ is $\{|G| = 1\}$ for some $G$.

We shall construct a lamination $\lambda_F$ whose support is contained in $S$, such that every $F$-calibrated hypersurface is a leaf of $\lambda_F$.
By Theorem \ref{disjoint surfaces are lamination}, we must establish the following:
\begin{enumerate}
\item There is an $F$-calibrated hypersurface. \label{plan existence}
\item There is a uniform bound on the curvatures of the $F$-calibrated hypersurfaces. \label{plan curvature}
\item Any two $F$-calibrated hypersurfaces are disjoint. \label{plan disjoint}
\item The limit of a sequence of $F$-calibrated hypersurfaces is a $F$-calibrated hypersurface.
\end{enumerate}
The nontriviality condition (\ref{plan existence}) is nothing more than Lemma \ref{calibrated means measured stretch}.
The hypersurface furnished by that lemma is actually calibrated by \emph{every} calibration in $\rho$, so the intersection of all laminations $\lambda_F$ is nonempty; this intersection shall be the canonical calibrated lamination.

We now show that the leaves of the putative canonical lamination satisfy the necessary curvature bounds; this is a little subtle because the leaves are injectively immersed but not embedded.
In the below lemmata, let $r_*$ be the minimum of the injectivity radius of $M$ and $\delta \|\Riem_M\|_{C^0}^{-1/2}$, where $\delta > 0$ is a dimensional constant to be determined later.
Let $\Sph^{d - 1}$ be the round sphere of dimension $d - 1$.

\begin{lemma}\label{transversality}
Let
$$U := \bigcup_{x \in M} \{\xi \in T_x M: 0 < |\xi| < r_*\},$$
and let $F: U \to M$ be the exponential map.
Then for every injectively immersed hypersurface $N \subset M$, $F$ is transverse to $N$.
\end{lemma}
\begin{proof}
We must show that for every $(x, \xi) \in U$ such that $F(x, \xi) \in N$, the image of
$$\dif F(x, \xi): T_{(x, \xi)} T_x M \to T_{F(x, \xi)} M$$
contains a vector not tangent to $N$.
Let $\eta$ be the unit normal to $N$ at $F(x, \xi)$, and let $\overline \eta$ be the parallel transport of $\eta$ along the unique geodesic $\gamma$ from $F(x, \xi)$ to $x$.
Viewing $\overline \eta$ as an element of $T_{F(x, \xi)} T_x M$, we see that if $\delta$ was chosen small enough, then $\dif F(x, \xi) \overline \eta$ lies in a small neighborhood of $\eta$.
Indeed, if $\delta$ was chosen small enough, then $\gamma$ is much shorter than the curvature scale $\|\Riem_M\|_{C^0}^{-1/2}$.
In particular, $\dif F(x, \xi) \overline \eta$ is not tangent to $N$.
\end{proof}

\begin{lemma}
For every calibration $F$, every complete injectively immersed $F$-calibrated hypersurface $N \subset M$, every $x \in M$, every $0 < r \leq r_*$, and every component $N'$ of $N \cap B(x, r)$,
\begin{equation}\label{area bound for calibrated}
\vol(\Ball^{d - 1}) r^{d - 1} \leq \vol(N') \leq 2\vol(\Sph^{d - 1}) r^{d - 1}.
\end{equation}
\end{lemma}
\begin{proof}
If $\delta$ was chosen small enough, then
$$\vol(\partial B(x, r)) < 2\vol(\Sph^{d - 1}) r^{d - 1}.$$
Let $F: U \to M$ be the exponential map as in Lemma \ref{transversality}, so that $F$ is transverse to $N$.
By putting polar coordinates on each tangent space, we may view $U$ as a fiber bundle,
$$M \times (0, r_*) \to U \to \Sph^{d - 1}.$$
By the Thom transversality theorem, for almost every $(x, r) \in M \times (0, r_*)$, the induced map
\begin{align*}
f_{x, r}: \Sph^{d - 1} &\to M \\
\omega &\mapsto F(x, r\omega)
\end{align*}
is transverse to $N$.
But $f_{x, r}$ is the embedding $\Sph^{d - 1} \to \partial B(x, r)$.
The estimate (\ref{area bound for calibrated}) is preserved by slight perturbations of $r$, so we may use the above considerations to reduce to the case that $N$ is transverse to $\partial B(x, r)$.

Let $N'$ be a component of $N \cap B(x, r)$, so that $N'$ is embedded (not just injectively immersed).
By transversality, $N' \cap \partial B(x, r)$ is diffeomorphic to a closed $d - 2$-dimensional submanifold of $\Sph^{d - 1}$.
Since $H_{d - 2}(\Sph^{d - 1}, \RR) = 0$, there exists a relatively open set $V \subseteq \partial B(x, r)$ which is bounded by $N \cap \partial B(x, r)$.
Because of how we chose $r_*$, we may use the $L^\infty$ Poincar\'e lemma, Theorem \ref{Hodge theorem}, to find a continuous $d - 2$-form $A$ on a neighborhood of the closure of $B(x, r)$, such that $F = \dif A$.
Then
\begin{align*}
\vol(N \cap B(x, r)) &= \int_{N \cap B(x, r)} F = \int_{N \cap \partial B(x, r)} A = \int_V F \leq \vol(V) \leq \vol(\partial B(x, r)) \\
&< 2\vol(\Sph^{d - 1}) r^{d - 1}. \qedhere
\end{align*}
\end{proof}

\begin{lemma}
There exists a constant $C > 0$, only depending on $M$, such that for every calibration $F$ and complete injectively immersed $F$-calibrated hypersurface $N$, we have the curvature bound
\begin{equation}\label{curvature bound for calibrated}
\|\Two_N\|_{C^0} \leq C.
\end{equation}
\end{lemma}
\begin{proof}
Let $x \in N$ and let $r > 0$ be small enough depending on $M$.
Then each component $N'$ of $N \cap B(x, r)$ is $F$-calibrated, and therefore a stable minimal hypersurface.
By (\ref{area bound for calibrated}), $\vol(N') \lesssim r^{d - 1}$.
So by \cite[pg785, Corollary 1]{Schoen81} (see also \cite[Chapter 2, \S\S4-5]{colding2011course}),
\begin{align*}
\|\Two_{N'}\|_{C^0(B(x, r/2))} \lesssim_{d, \|\Riem_g\|_{C^0(B(x, 2r))}} \frac{1}{r}.
\end{align*}
Since $N'$ was an arbitrary component, the same estimate holds for $N$.
Using the compactness of $M$, we may cover it by finitely many balls in which estimates of this form hold to conclude (\ref{curvature bound for calibrated}).
\end{proof}

\begin{lemma}\label{calibrated implies disjoint}
Let $F$ be a calibration, and let $N, N'$ be immersed $F$-calibrated hypersurfaces. Then:
\begin{enumerate}
\item If $N \cap N'$ is nonempty, then for each $x \in N \cap N'$ there is an open neighborhood $U$ of $x$ such that $N \cap U = N' \cap U$. \label{disjoint:locally unique}
\item If $N \cap N'$ is nonempty, and $N, N'$ are complete and connected, then $N = N'$. \label{disjoint:unique}
\item $N$ is injectively immersed. \label{disjoint:injective}
\end{enumerate}
\end{lemma}
\begin{proof}
We first observe that for each $x \in N$, $(\star F(x))^\sharp$ is the (unique) normal vector to $N$ at $x$ (and similarly for $N'$), and so if $x \in N \cap N'$ then $N, N'$ have the same tangent space at $x$.
So for each $x \in N \cap N'$, there exists $r > 0$ and normal coordinates $(\xi, \eta) \in \RR^{d - 1} \times \RR$ on $B(x, r)$ based at $x$, such that for each pair of sheets $N_* \subseteq N \cap B(x, r)$, $N_*' \subseteq N' \cap B(x, r)$ which contain $x$, there exists a relatively open set $V \subseteq \{\eta = 0\}$, an open set $U \subseteq B(x, r)$ containing $x$, and functions $u, u': V \to \RR$ such that:
\begin{prfenum}
\item $N_* \cap U = \{(\xi, u(\xi)): \xi \in V\}$.
\item $N_*' \cap U = \{(\xi, u'(\xi)) : \xi \in V\}$.
\item $u(0) = u'(0) = 0$.
\item If $u(\xi) = u'(\xi)$ then $\dif u(\xi) = \dif u'(\xi)$.
\end{prfenum}
Let $v := u - u'$. Then:
\begin{prfenum}
\item $v(0) = 0$.
\item $v$ satisfies a linear elliptic PDE on $V$ \cite[Proof of Theorem 7.3]{colding2011course}.
\item If $v(\xi) = 0$ then $\dif v(\xi) = 0$.
\end{prfenum}
We claim that $v$ is identically $0$.
If this is not true, the set $\{v = 0\} = \{v = \dif v = 0\}$ is $d - 3$-rectifiable \cite[Lemma 1.9]{Hardt89}, but $\dim V = d - 1$, so $\{v \neq 0\}$ is connected.
So either $v \geq 0$ or $v \leq 0$, and $v$ has a zero; this contradicts the maximum principle.

The above discussion shows that $N_* \cap U = N_*' \cap U$.
Taking $N = N'$ we see that $N$ only has one sheet in $B(x, r)$ which contains $x$, so (\ref{disjoint:injective}) holds.
So running the same argument, without assuming that $N = N'$, yields (\ref{disjoint:locally unique}).
A continuity argument then implies (\ref{disjoint:unique}).
\end{proof}

We must show that a limit of $F$-calibrated hypersurfaces is $F$-calibrated, and to make this precise we shall need the notion of a \dfn{Vietoris limit superior} of a sequence of closed sets, \cite[\S4.F]{kechris2012classical}.
If $(K_n)$ is a sequence of closed subsets of $M$, then $\limsup_{n \to \infty} K_n$ is the set of all $x$ such that for every open set $U \ni x$, there exist infinitely many $n \in \NN$ such that $U \cap K_n \neq \emptyset$; one easily checks that $\limsup_{n \to \infty} K_n$ is closed.

\begin{lemma}\label{limit of calibrateds is calibrated}
Let $F$ be a calibration, let $(N_n)$ be a sequence of $F$-calibrated complete connected immersed hypersurfaces, and let $K := \limsup_{n \to \infty} \overline{N_n}$.
For every $x \in K$ there exists a $F$-calibrated complete connected immersed hypersurface $N \subseteq K$ such that $x \in N$.
\end{lemma}
\begin{proof}
By taking a subsequence, we may assume that there exist $x_n \in N_n$ such that $x_n \to x$.
By Lemma \ref{calibrated implies disjoint}, we may also assume that if $N_n \cap N_m$ is nonempty then $n = m$.
Combining this with the curvature bound (\ref{curvature bound for calibrated}), we obtain the hypotheses of \cite[Lemma 2.4]{BackusCML}.
The conclusion of that lemma is that for every $\delta > 0$ there exists $r > 0$ only depending on $M$, and normal coordinates $(\xi, \eta) \in \RR^{d - 1} \times \RR$ on $B(x, r)$ based at $x$ such that for every $n$,
\begin{equation}\label{approximate normal vector by vertical}
\|\normal_{N_n} - \partial_\eta\|_{C^0(B(x, r))} \leq \delta.
\end{equation}
If $\delta$ was chosen small enough, depending only on $M$, then by the vertical line test, there exists a relatively open set $V \subseteq \{\eta = 0\}$ and a sequence of functions $u_n$ on $V$, such that:
\begin{prfenum}
\item $x \in V$.
\item There exists $c_0 > 0$ which only depends on $M$ such that $\diam(V) \geq c_0$.
\item For every $n$, $N_n \cap \{(\xi, \eta) \in B(x, r): \xi \in V\} = \{(\xi, u_n(\xi)): \xi \in V\}$.
\end{prfenum}
The functions $u_n$ solve the minimal surface equation, 
$$Pu(\xi) := F(\xi, u(\xi), \dif u(\xi), \nabla^2 u(\xi)) = 0$$
where one can use \cite[(7.21)]{colding2011course} to show that $F$ has the form 
$$F(\xi, \eta, A, B) := \tr B + O((|\xi| + |\eta| + |A|)(1 + |B|))$$
where the implied constant only depends on $M$.
But $|\xi| + |u_n(\xi)| \lesssim r$ and, if 
$$\|\normal_{N_n} - \partial_\eta\|_{C^0(B(x, r))} \leq \frac{1}{10},$$
then one may show that
$$|\dif u_n(\xi)| \leq \|\dif u_n\|_{C^0} \lesssim \|\normal_{N_n} - \partial_\eta\|_{C^0(B(x, r))}.$$
So by (\ref{approximate normal vector by vertical}), we can first choose $\delta$ small enough depending on $M$, and then choose $r$ small enough depending on $\delta$, so that for every $n$ large enough depending on $r$, the equation $Pu_n = 0$ is uniformly elliptic.
In particular, by the interior Schauder estimate \cite[Theorem 6.2]{gilbarg2015elliptic}, we may choose $\delta$ small enough, depending only on $M$, that there exists a connected, relatively open set $W \subseteq V$ such that:
\begin{prfenum}
\item $x \in W$. \label{W prop 1}
\item There exists $c_1 > 0$ which only depends on $M$ such that $\diam(W) \geq c_1$. \label{W prop 2}
\item For every sufficiently large $n$, $\|u_n\|_{C^3(W)} \leq 1$.
\end{prfenum}
Therefore there exists $u \in C^2(W)$ such that:
\begin{prfenum}
\item After passing to a subsequence, $u_n \to u$ in $C^2(W)$.
\item $N^x := \{(\xi, u(\xi)): \xi \in W\}$ contains $x$.
\item $N^x \subseteq K$.
\end{prfenum}
We moreover claim that, possibly after shrinking $W$ (while preserving \ref{W prop 1} and \ref{W prop 2}):
\begin{prfenum}
\item $N^x$ is $F$-calibrated. \label{local limit surface is calibrated}
\item $N^x$ is geodesically convex. \label{local limit surface is convex}
\item There exists $c_2 > 0$ which only depends on $M$ such that $\dist_{N^x}(x, \partial N^x) \geq c_2$. \label{local limit surface has distance comparison}
\end{prfenum}
To prove this, let $N_n^x := \{(\xi, u_n(x)): \xi \in W\}$.
If $\diam(W)$ was chosen small enough (depending only on $M$), then we can use the $L^\infty$ Poincar\'e lemma, Theorem \ref{Hodge theorem}, to find a continuous $d - 2$-form $A$ on $W$ such that $\dif A = F$.
Since $u_n \to u$ in $C^2(W)$, we can compute using Stokes' theorem
$$\int_{N^x} F = \int_{\partial N^x} A = \lim_{n \to \infty} \int_{\partial N_n^x} A = \lim_{n \to \infty} \int_{N_n^x} F = \lim_{n \to \infty} \vol(N_n^x) = \vol(N^x),$$
which proves \ref{local limit surface is calibrated}.
By shrinking $W$ slightly more, we can impose \ref{local limit surface is convex}.
Moreover, since $\partial N^x \subset \partial W$, and the curvature bound (\ref{curvature bound for calibrated}) allows us to compare distances in $M$ and distances in $N^x$, \ref{local limit surface has distance comparison} follows from \ref{W prop 2}.

Let $N$ be the union of all $F$-calibrated connected immersed hypersurfaces contained in $K$ which extend $N^x$.
If $N$ is incomplete, then there exists $y \in N$ such that $\dist_N(y, \partial N) < c_2$.
Then $y \in K$, so there exists a $F$-calibrated connected immersed hypersurface $N^y \subseteq K$ such that $y \in N^y$ and \ref{local limit surface is convex} and \ref{local limit surface has distance comparison} hold.
But $N^y \subseteq N$, so by \ref{local limit surface is convex} and \ref{local limit surface has distance comparison},
$$\dist_N(y, \partial N^y) \geq \dist_N(y, \partial N^y) \geq c_2,$$
which is a contradiction. Therefore $N$ is complete.
\end{proof}

\begin{lemma}\label{existence of semicanonical lamination}
Let $F$ be a calibration in $\rho$.
Then the set of $F$-calibrated connected complete immersed hypersurfaces is the set of leaves of a lamination $\lambda_F$, which contains every measured oriented $F$-calibrated lamination.
\end{lemma}
\begin{proof}
Let $\mathscr L_F$ be the set of connected complete immersed $F$-calibrated hypersurfaces.
By Lemma \ref{calibrated implies disjoint}, $\mathscr L_F$ consists of pairwise disjoint injectively immersed minimal hypersurfaces.
The curvature bound (\ref{curvature bound for calibrated}) only depends on $M$, and implies that the elements of $\mathscr L_F$ have curvatures bounded uniformly in $C^0$.
By Lemma \ref{calibrated means measured stretch}, $\mathscr L_F$ is nonempty.

Let $E$ be the union of all elements of $\mathscr L_F$.
If $(x_n)$ is a sequence in $E$, say $x_n \in N_n$ for some $N_n \in \mathscr L_F$, and $x_n \to x$, then $x \in \limsup_{n \to \infty} \overline{N_n}$.
So by Lemma \ref{limit of calibrateds is calibrated}, there exists $N \in \mathscr L_F$ such that $x \in N$.
In particular, $x \in E$, so $E$ is closed.

By the above discussion and Theorem \ref{disjoint surfaces are lamination}, $\mathscr L_F$ is the set of leaves of some lamination $\lambda_F$.
\end{proof}

\begin{proof}[Proof of Theorem \ref{existence of calibrated lam}]
Let $S$ be the set of calibrations in $\rho$, which is nonempty since $\|\rho\|_\infty = 1$.
Then there is a lamination which is $F$-calibrated by every $F \in S$.
Indeed, by Lemma \ref{calibrated means measured stretch}, there is a measured oriented lamination $\kappa$ which is $F$-calibrated for some $F \in S$, and by Lemma \ref{properties of calibrated laminations}, $\kappa$ is $F$-calibrated for every $F \in S$.

For every $F \in S$, let $\lambda_F$ be the calibrated lamination produced by Lemma \ref{existence of semicanonical lamination}.
By Lemma \ref{existence of intersections}, there is a lamination $\lambda_\rho$ whose set of leaves is the intersection over $F \in S$ of the sets of leaves of $\lambda_F$. 
Then $\lambda_\rho$ has all desired properties.
\end{proof}

\section{Transverse measures on the canonical lamination}\label{canonical structure}
\subsection{Ergodic theory of \texorpdfstring{$\lambda_\rho$}{the canonical lamination}}
Let $M$ be a closed oriented Riemannian manifold of dimension $2 \leq d \leq 7$.
For each oriented lamination $\lambda$ in $M$, let $\mathcal M(\lambda)$ be the set of transverse probability measures to sublaminations of $\lambda$.
This set inherits the vague topology on the space of Borel probability measures on $\supp \lambda$; in view of (\ref{decomposition of Ruelle Sullivan}), this topology is the same as the topology on the space of measured laminations (see \cite{BackusCML}) restricted to $\mathcal M(\lambda)$.
It is clear that $\mathcal M(\lambda)$ is convex, and one may use the compactness of the space of Borel probability measures on the compact metrizable space $\supp \lambda$ \cite[Theorem 17.23]{kechris2012classical} to show that $\mathcal M(\lambda)$ is compact.
By the Krein-Milman theorem, if $\mathcal M(\lambda)$ is nonempty, then so is its set of extreme points, $\mathcal E(\mathcal M(\lambda))$.

\begin{definition}
A measure $\mu \in \mathcal M(\lambda)$ is \dfn{ergodic} if, for every saturated set $E$, either $\mu(E) = 0$ or $\mu(E) = 1$.
\end{definition}

\begin{lemma}
Every extreme point of $\mathcal M(\lambda)$ is ergodic, and the set of ergodic measures is linearly independent in the space of signed Borel measures on $\supp \lambda$.
\end{lemma}
\begin{proof}
The first claim is an easy modification of the proof of \cite[Theorem 4.4]{einsiedler2010ergodic}, and the second is essentially the proof that every ultrafilter on a finite set is principal.
To be more precise, let $S$ be a finite set of ergodic measures, and choose $c_\mu \in \RR$ such that $\sum_{\mu \in S} c_\mu\mu = 0$.
The measures in $S$ are determined by their values on saturated sets, so if some coefficient $c_\nu$ is nonzero, then there exists a saturated set $E$ and a proper subset $T \subset S$ such that:
\begin{prfenum}
\item $\nu \in T$.
\item For every $\mu \in T$, $\mu(E) = 1$.
\item For every $\mu \in S \setminus T$, $\mu(E) = 0$.
\end{prfenum}
Then $S' := \{1_E \mu: \mu \in T\}$ satisfies:
\begin{prfenum}
\item For every $\mu \in S'$, $\mu$ is ergodic. 
\item $\sum_{\mu \in S'} c_\mu \mu = 0$.
\item There exists $\nu \in S'$ such that $c_\nu$ is nonzero.
\item $\card S' < \card S$.
\end{prfenum}
Therefore we can repeat the argument with $S$ replaced by $S'$.
After finitely many iterations, we reduce to the case that $\card S \leq 1$, in which case we have a contradiction.
\end{proof}

Now let $B$ be the stable unit ball of $H_{d - 1}(M, \RR)$ and $B^*$ be the costable unit ball.
For each $\rho \in B^*$, we consider the set $\mathcal M(\lambda_\rho)$ of transverse probability measures to the canonical lamination, $\lambda_\rho$.
By Lemma \ref{properties of calibrated laminations}, $\mathcal M(\lambda_\rho)$ is the set of measured oriented laminations which are calibrated by some calibration in $\rho$.
The map which sends a measured oriented lamination to its homology class restricts to a an affine map $\mathcal M(\lambda_\rho) \to \rho^*$, which is surjective by Lemma \ref{calibrated means measured stretch}.
In particular, if $\alpha$ is an extreme point of the dual flat, $\rho^*$, then $\alpha$ is the homology class of an ergodic measure.

We summarize the above discussion as Corollary \ref{extreme points are indecomposable}.

\subsection{The Arnoux--Levitt lemma}
Following an idea of Auer and Bangert \cite{Auer01}, we study $\mathcal M(\lambda_\rho)$ using an ergodic-theoretic lemma of Arnoux and Levitt, \cite[Proposition 3.1]{Arnoux1986}.
We need a more general version of the Arnoux--Levitt lemma, and the original proof is in French, so we include a proof here.

Let us identify transverse measures with positive transverse cocycles (cocycles which act on curves transverse to the lamination and are cooriented with the lamination); this is standard, and we refer to \cite[\S7.2]{daskalopoulos2020transverse} for a justification of this identification.

\begin{lemma}\label{equality of measures}
Let $\lambda$ be an oriented lamination, $U \subseteq M$ open, and $\mu, \nu \in \mathcal M(\lambda)$.
Assume that:
\begin{enumerate}
\item $\lambda$ is not a closed hypersurface, and there is a leaf of $\lambda$ which is dense in $\supp \lambda \cap U$. \label{msrequality:irrational}
\item $\mu(U) = \nu(U) = 1$. \label{msrequality:ergodic}
\item There exists $b \geq 0$ such that for every $1$-cycle $C \subset U$ which is transverse to $\lambda$, $\mu(C) - \nu(C) \in b\ZZ$. \label{msrequality:integrality}
\end{enumerate}
Then $\mu = \nu$.
\end{lemma}
\begin{proof}
We follow \cite[\S3, Lemme]{Arnoux1986} which is a similar result when $\lambda$ is a minimal component of a foliation.

Let $C \subset U$ be a transverse curve to $\lambda$, which is cooriented with $\lambda$, and such that $\mu(C) > 0$ and $\nu(C) > 0$.
By assumptions (\ref{msrequality:irrational}) and (\ref{msrequality:ergodic}), there exists a leaf $N$ such that:
\begin{prfenum}
\item $N$ is dense in $\supp \lambda \cap U$.
\item $N \cap C$ is infinite and dense in $\supp \lambda \cap U \cap C$.
\end{prfenum}
If we prove that $\mu(C) = \nu(C)$ for every sufficiently short cooriented transverse curve $C$, then the result follows for all curves by $\sigma$-additivity.
Therefore we may shorten $C$ so that:
\begin{prfenum}
\item $C$ begins and ends on $N$.
\item If $b > 0$ then $\mu(C) + \nu(C) < b$.
\end{prfenum}
Let $\sigma \subset N$ be a curve from the beginning of $C$ to the end of $C$, and let $C'$ be a deformation of $C \cup \sigma$ through homotopies which leave $\lambda \setminus N$ fixed, so that $C'$ is a transverse cycle to $\lambda$.
Then, by (\ref{msrequality:integrality}), for some $k \in \ZZ$,
$$\mu(C) = \mu(C') = \nu(C') + kb = \nu(C) + kb.$$
If $b = 0$ then we are done; otherwise, since $\mu(C) + \nu(C) < b$, it follows that $k = 0$.
\end{proof}

\begin{lemma}[Arnoux--Levitt lemma]\label{abelian case}
Let $\lambda$ be an oriented lamination, and let $\mathcal I(\lambda)$ be the set of ergodic probability measures transverse to sublaminations $\kappa \subseteq \lambda$ such that $\kappa$ is not a closed hypersurfaces.
Then the homology classes of measures in $\mathcal I(\lambda)$ are linearly independent, and if they span $H_{d - 1}(M, \RR)$, then $H_{d - 1}(M, \RR) = 0$.
\end{lemma}
\begin{proof}
Let $\kappa_1, \dots, \kappa_q$ be distinct sublaminations of $\lambda$, such that for each $1 \leq i \leq q$:
\begin{prfenum}
\item There exists $m_i \geq 1$ and distinct probability measures $\mu_i^1, \dots, \mu_i^{m_i}$ such that for each $1 \leq j \leq m_i$, $\mu_i^j$ is ergodic and transverse to $\kappa_i$.
\item $\kappa_i$ is not a closed hypersurface.
\end{prfenum}
Notice that if $q > 0$ then $H_{d - 1}(M, \RR) \neq 0$.
We define open sets $U_i$ and leaves $N_i$ of $\kappa_i$, such that:
\begin{prfenum}
\item $N_i$ is dense in $\supp \kappa_i \cap U_i$.
\item For every $j$, $\mu_i^j(U_i) = 1$.
\end{prfenum}
To do this we break into cases:
\begin{prfenum}
\item Suppose that $\kappa_i$ is not a foliation with a dense leaf. By the Morgan--Shalen decomposition (Theorem \ref{MorganShelan}) and the ergodicity of $\mu_1^i$, $\kappa_i$ is exceptional, and if $i' \neq i$ then $\supp \kappa_{i'}$ avoids an open set $U_i$ containing $\supp \kappa_i$. Let $N_i$ be any leaf of $\kappa_i$.
\item Suppose that $\kappa_i$ is a foliation with a dense leaf $N_i$. In particular $\kappa_i = \lambda$. For any $i' \neq i$, and any $j$, $\mu_i^j$ is ergodic and $\supp \kappa_{i'}$ is saturated, so $\mu_i^j(\supp \kappa_{i'}) = 0$. Therefore $U_i := M \setminus \bigcup_{i' \neq i} \supp \kappa_i$ has $\mu_i^j(U_i) = 1$. Since $\kappa_{i'}$ is not a foliation, $\supp \kappa_{i'}$ is a saturated closed subset of $M$ which is not $M$; therefore it misses the dense leaf $N_i$.
\end{prfenum}
Next we show:
\begin{prfenum}
\item For every $i$, $([\mu_i^j])_j$ is linearly independent. \label{abelian:irrational means independent}
\end{prfenum}
Suppose that there are $a_j \in \RR$ such that $\sum_j a_j [\mu^i_j] = 0$, let $\mu$ be the sum of $a_j \mu_i^j$ over $j$ such that $a_j > 0$, and let $\nu$ be the sum of $-a_j \mu_i^j$ over $j$ such that $a_j < 0$.
Then $[\mu] - [\nu] = 0$, so by Lemma \ref{equality of measures} with $b = 0$ and $U = M$, $\mu - \nu = 0$, hence $\sum_j a_j \mu^i_j = 0$.
By ergodicity, $(\mu_i^j)_j$ is linearly independent, so $a_j = 0$, establishing \ref{abelian:irrational means independent}.

We claim there are $1$-cocycles $t_i$ such that:
\begin{prfenum}
\item There exists a $1$-cycle $C_i \subset U_i$ such that $t_i(C_i) \neq 0$.
\item For every $j \neq i$, and every $1$-cycle $C \subset U_j$, $t_i(C) = 0$. \label{abelian:rational independent}
\item There exists $q_i \in \QQ$ such that for every $1$-cycle $C$, $t_i(C) \in q_i \ZZ$. \label{abelian:rational independent 2}
\end{prfenum}
By composing with the natural homomorphism $H_1(M, \ZZ) \to H_1(M, \RR)$, we can think of the cohomology class of $\mu_i^1$ as a homomorphism
$$[\mu_i^1]: H_1(M, \ZZ) \to \RR.$$
Since $M$ is compact, $H_1(M, \ZZ)$ is finitely generated and so we can slightly perturb the value of $[\mu_i^1]$ on the generators to obtain $t_i$ with the desired properties.

To complete the proof it is enough to show that
\begin{prfenum}
\item $(t_i, [\mu_i^j])_{i, j}$ is linearly independent. \label{abelian:irrational means independent 2}
\end{prfenum}
Suppose that there are $a_i^j, a_i \in \RR$ such that 
$$\sum_{i, j} a_i^j [\mu_i^j] + \sum_i a_i t_i = 0.$$
Then for every $1$-cycle $C$ in $U_i$,
$$\sum_j a_i^j \mu_i^j(C) = - a_i t_i(C) \in a_i q_i \ZZ.$$
Let $\mu_i^+$ be the sum of $a_i^j \mu_i^j$ taken over $j$ such that $a_i^j > 0$, and let $\mu_i^-$ be the sum of $a_i^j \mu_i^j$ taken over $j$ such that $a_i^j < 0$.
Then $\mu_i^+ - \mu_i^- \in a_i q_i \ZZ$, so by Lemma \ref{equality of measures}, $\mu_i^+ = \mu_i^-$, or in other words $\sum_j a_i^j [\mu_i^j] = 0$.
So by \ref{abelian:irrational means independent}, $a_i^j = 0$, so $\sum_i a_i t_i = 0$.
By \ref{abelian:rational independent} and \ref{abelian:rational independent 2}, $(t_i)_i$ is linearly independent, so $a_i = 0$, establishing \ref{abelian:irrational means independent 2}.
\end{proof}

\begin{theorem}\label{ergodic decomposition thm}
Let $\lambda$ be an oriented lamination which is not a foliation with a dense leaf.
Then $\mathcal E(\mathcal M(\lambda))$ is the set of ergodic measures transverse to sublaminations of $\lambda$.
\end{theorem}
\begin{proof}
We have already seen that every extreme point of $\mathcal M(\lambda)$ is an ergodic measure and so we just need to show the converse.
By the Morgan--Shalen decomposition, Theorem \ref{MorganShelan}, we may reduce to the case that $\lambda$ either has only closed leaves, or has no closed leaves.
If $\lambda$ has only closed leaves, then every ergodic sublamination is a single leaf, and in particular cannot be written as a convex combination of any other ergodic sublaminations.
Otherwise, $\lambda$ has no closed leaves, so by Lemma \ref{abelian case}, $\mathcal E(\mathcal M(\lambda))$ is finite, say $\{\mu_1, \dots, \mu_m\}$.
For each $i \neq j$, there exists a saturated set $E_{ij}$ such that $\mu_i(E_{ij}) = 1$ but $\mu_j(E_{ij}) = 0$.
In particular, if $\mu$ is not an extreme point and we write $\mu = \sum_i d_i \mu_i$, then there exist $i \neq j$ such that $d_i, d_j > 0$.
Then $d_i \leq \mu(E_{ij}) \leq 1 - d_j$, so $\mu$ is not ergodic.
\end{proof}

\subsection{The dual flat \texorpdfstring{$\rho^*$}{of a cohomology class}}
Let $b_1 := \dim H_1(M, \RR)$ be the first Betti number of $M$, and let $\rho \in \partial B^*$ be a costable unit class.
We are going to prove Theorem \ref{vertex counting}: \emph{$\rho^*$ is a polytope, vertices of $\rho^*$ have rational direction iff they are represented by closed leaves of $\lambda_\rho$, and $\rho^*$ has at most $b_1 - 1$ vertices of irrational direction}, Corollary \ref{strict convexity iff unique ergodicity}: \emph{if $B$ is strictly convex, then every ergodic calibrated lamination is uniquely ergodic}, and Theorem \ref{Auer Bangert thm}(\ref{ABt1}): \emph{if $S \subset \partial B$ is flat and $\alpha, \beta \in S$, then their intersection product $\alpha \cdot \beta$ vanishes}.

\begin{lemma}\label{rational implies closed}
Let $F$ be a calibration and let $\lambda$ be an ergodic, $F$-calibrated, measured oriented lamination.
The following are equivalent:
\begin{enumerate}
\item $[\lambda]$ has rational direction.
\item $\lambda$ is a closed hypersurface.
\end{enumerate}
\end{lemma}
\begin{proof}
If $\lambda$ is a closed hypersurface $N$, then $[\lambda]$ is a rescaling of $[N]$, and $[N]$ is the image of the class of $N$ in $H_{d - 1}(M, \ZZ)$.

Conversely, assume that $[\lambda]$ has rational direction.
Since $\pi_1(M)$ is finitely generated, we may rescale $M$ suitably so that $[\lambda]$ is a representation $\alpha: \pi_1(M) \to \ZZ$.
Since such representations are identified with homotopy classes of maps $M \to \Sph^1$, the Ruelle-Sullivan current $T_\lambda$ takes the form $\dif u$ for some map $u: M \to \Sph^1$.
Let $\tilde u \in BV_\loc(\tilde M, \RR)$ be the universal cover of $u$.

Towards contradiction, let $N$ be a leaf of $\lambda$ which is not closed, and let $\tilde N \subset \tilde M$ be the preimage of $N$.
Since $N$ is not closed and $M$ is compact, there exists $x \in N$ such that $N$ accumulates on itself at $x$, in the sense that for every sufficiently small $r > 0$, $N \cap B(x, r)$ has infinitely many connected components.
Let $\tilde x \in \tilde M$ be a point in the preimage of $x$.
Thus the set $T$ of $t \in \RR$ such that $\partial \{\tilde u > t\}$ intersects $B(\tilde x, r/2)$ is infinite.

We claim that there exists $c > 0$ such that for any $t \in \RR$ such that $\partial \{\tilde u > t\}$ intersects $B(\tilde x, r/2)$,
$$\vol(\partial \{\tilde u > t\} \cap B(\tilde x, r)) \geq c.$$
To see this, let $\tilde y \in \partial \{u > t\} \cap B(\tilde x, r/2)$.
Since $\partial \{\tilde u > t\}$ is smooth, its density $\theta$ (in the sense of rectifiable sets) at $\tilde y$ is the volume of the unit ball of $\RR^{d - 1}$.
By the monotonicity formula for minimal hypersurfaces \cite[Theorem 7.11]{Marques}, there exists $A \geq 0$ which only depends on $M$ such that for any $\rho > 0$,
$$\vol(\partial \{\tilde u > t\} \cap B(\tilde y, \rho)) \geq e^{-A\rho^2} \theta \rho^{d - 1}$$
and the claim follows by taking $c := e^{-Ar^2/4} \theta$ and $\rho := r/2$.

The image of $T$ in $\Sph^1$ is a point, so for any $t, s \in T$, either $t = s$ or $|t - s| \geq 1$.
We may assume that there is an infinite increasing sequence $(t_n)$ in $T$.
By the coarea formula \cite[Theorem 1.23]{Giusti77},
$$\int_{B(x, r)} \star |\dif \tilde u| \geq \sum_{n=0}^\infty \int_{t_n}^{t_{n + 1}} \vol(\partial \{u > t\}) \dif t \geq c \sum_{n=0}^\infty (t_{n + 1} - t_n) = \infty,$$
which is a contradiction, since $\tilde u \in BV_\loc(\tilde M, \RR)$.

So if $[\lambda]$ has rational direction, then every leaf of $\lambda$ is closed.
Since $\lambda$ is ergodic, it follows that $\lambda$ is a single closed leaf.
\end{proof}

\begin{lemma}\label{lma: linear dependence means disconnection}
Let $\mathcal S$ be a finite set of disjoint closed hypersurfaces in $M$.
If $M \setminus \bigcup \mathcal S$ is connected, then the set of homology classes of $\mathcal S$ is linearly independent over $\RR$.
\end{lemma}
\begin{proof}
Assume towards contradiction that there is a $d$-chain $\sigma$, $N_1, \dots, N_k \in \mathcal S$, and $c_1, \dots, c_k \in \RR \setminus \{0\}$, such that $\partial \sigma = \sum_{j \leq k} c_j N_j$.
We identify $\sigma$ with the function on the open set $U := M \setminus \bigcup \mathcal S$ which sends a point $x$ to the multiplicity of $\sigma$ at $x$.
Then $\sigma$ is locally constant, but $U$ is connected, so $\sigma$ is actually constant.
Therefore $\partial \sigma = 0$, so $\{N_1, \dots, N_k\}$ is linearly dependent when viewed as a set of $d - 1$-chains.
So there exists $2 \leq j \leq k$ such that $N_1 \cap N_j \neq \emptyset$, a contradiction.
\end{proof}

\begin{proof}[Proof of Theorem \ref{vertex counting}]
Let $\Pi: \mathcal M(\lambda_\rho) \to \partial B$ be the map which sends a transverse probability measure to its homology class.
Let $m$ be the number of measures on sublaminations which are not closed hypersurfaces.
By Lemma \ref{rational implies closed}, $\Pi$ sends closed hypersurfaces to rational extreme points of $\rho^*$ and all other measures to irrational extreme points.
By Corollary \ref{extreme points are indecomposable}, $\Pi$ is surjective, so the number of irrational extreme points is at most $m$.
By Lemma \ref{abelian case}, $m \leq \max(0, b_1 - 1)$.

We claim that $\rho^*$ has only finitely many rational extreme points.
If not, there are infinitely many closed leaves $N_n$ of $\lambda_\rho$ with distinct homology classes $\alpha_n \in \mathcal E(\rho^*)$.
Then $\{\alpha_1, \dots, \alpha_{b_1 + 1}\}$ is linearly dependent, so by Lemma \ref{lma: linear dependence means disconnection}, $U := M \setminus \bigcup_{n \leq b_1 + 1} N_n$ must be disconnected.
So every noncompact leaf $N$ of $\lambda_\rho$ must miss a connected component of the open set $U$; it follows that $N$ is not dense in $M$.
In particular, $\lambda_\rho$ is not a foliation with a dense leaf.

Let $\mu \in \mathcal M(\lambda_\rho)$ be defined by $\mu(N_n) = 2^{-n}$ (so $\mu(M \setminus \bigcup_n N_n) = 0$).
Then $\supp \mu$ is a sublamination $\kappa$ of $\lambda_\rho$ in which $\bigcup_n N_n$ is dense, and $\kappa$ is not a foliation with a dense leaf.
So by the Morgan--Shalen decomposition (Theorem \ref{MorganShelan}), $\kappa$ is the disjoint union of finitely many parallel families of closed leaves.
By the pigeonhole principle, there exist $m < n$ such that $N_m$ and $N_n$ are in the same parallel family.
Therefore $\alpha_m = \alpha_n$, a contradiction; this establishes the claim.

Therefore $\rho^*$ is a closed convex set with only finitely many extreme points.
So $\rho^*$ is a convex polytope.
\end{proof}

\begin{proof}[Proof of Corollary \ref{strict convexity iff unique ergodicity}]
Suppose that $B$ is strictly convex, and let $(\kappa, \mu)$ be an ergodic lamination which is $F$-calibrated for some calibration $F$.
We may assume that $\kappa$ is not a closed hypersurface, since closed hypersurfaces are uniquely ergodic.
Let $\alpha$ be the homology class of $(\kappa, \mu)$ and let $\rho$ be the cohomology class of $F$.
By Lemma \ref{properties of calibrated laminations}, $\mu \in \mathcal M(\lambda_\rho)$, so by Corollary \ref{extreme points are indecomposable} and strict convexity of $B$, $\rho^* = \{\alpha\}$.
Therefore, by Lemma \ref{abelian case}, $\mathcal M(\lambda_\rho) = \{\mu\}$, so $\kappa$ is uniquely ergodic.
\end{proof}

\begin{proof}[Proof of Theorem \ref{Auer Bangert thm}(\ref{ABt1})]
There exists $\rho \in \partial B^*$ such that $S \subseteq \rho^*$.
By Lemma \ref{calibrated means measured stretch}, there exist measured sublaminations $\kappa_\alpha, \kappa_\beta$ of $\lambda_\rho$, of classes $\alpha, \beta$.
Let $\dif u_\alpha, \dif u_\beta$ be their Ruelle-Sullivan currents, and suppose that $x$ is in the union of their supports.
If $N$ denotes the leaf of $\lambda_\rho$ containing $x$, then for $\sigma = \alpha, \beta$,
$$\dif u_\sigma(x) = \normal_N^\flat(x) \mu_\sigma(x)$$
where $\mu_\sigma$ is given by (\ref{decomposition of Ruelle Sullivan}).
In particular, $\dif u_\alpha|_{\supp \dif u_\beta}$ is a (possibly distributional) scalar field times $\dif u_\beta$, so $\dif u_\alpha \wedge \dif u_\beta = 0$, hence $\alpha \cdot \beta = 0$.
\end{proof}
\subsection{Strict convexity and the derived series of \texorpdfstring{$\pi_1(M)$}{the fundamental group}}
Let $\Gamma := \pi_1(M)$ and let $(\Gamma^{(n)})$ be the derived series of $\Gamma$.
We now prove Theorem \ref{Auer Bangert thm}(\ref{ABt2}): \emph{if $\Gamma^{(1)}/\Gamma^{(2)}$ is a torsion group, then the stable unit ball $B$ is strictly convex}.

If $\hat M \to M$ is a Galois covering space, let $\Gal(\hat M, M)$ be the Galois group of deck transformations of $\hat M \to M$, so $\Gamma = \Gal(\tilde M, M)$.
The \dfn{universal abelian covering space} of $M$, $\tilde M^{\rm ab} \to M$, is the Galois covering space such that
\begin{equation}\label{Galois group is homology}
\Gal(\tilde M^{\rm ab}, M) = \frac{\Gamma}{\Gamma^{(1)}} = H_1(M, \ZZ).
\end{equation}
Since $\RR$ is abelian, we have a natural isomorphism
\begin{align*}
\Hom(\Gamma, \RR) &= \Hom(\Gamma/\Gamma^{(1)}, \RR)
\end{align*}
and every $\alpha$-equivariant function $u$ on $\tilde M$ descends to an $\alpha$-equivariant function $u^{\rm ab}$ on $\tilde M^{\rm ab}$.
Since $\RR$ is abelian and torsion-free, and $\Gamma^{(1)} = \pi_1(\tilde M^{\rm ab})$, $\Gamma^{(1)}/\Gamma^{(2)}$ is a torsion group iff
$$H^1(\tilde M^{\rm ab}, \RR) = \Hom(\Gamma^{(1)}, \RR) = 0.$$

The next two lemmata appeared in \cite{Auer12}, though the proof of Theorem \ref{Auer Bangert thm}(\ref{ABt2}) does not.
Since this manuscript is not publicly available, or complete, we reproduce them here with full credit to the original authors.

\begin{lemma}[{\cite{Auer12}}]\label{superlevel sets are connected}
Let $u$ be an $\alpha$-equivariant function of least gradient on $\tilde M$.
Then the set $\{u^{\rm ab} > t\}$ is connected.
\end{lemma}
\begin{proof}
Suppose that $\{u^{\rm ab} > t\}$ is disconnected.
Then $\alpha$ is nonzero: if $\alpha = 0$, then $u$ descends to a function of least gradient on $M$, which is constant since $M$ is closed, and then $\{u^{\rm ab} > t\}$ is either empty or $M$, a contradiction.

Let $F$ be a fundamental domain of $M$ in $\tilde M^{\rm ab}$.
Since $u^{\rm ab} \in L^\infty_\loc$ \cite[Theorem 4.3]{Gorny20} and $F$ is compact, there exists $t_0 \in \RR$ such that $u > t_0$ on $F$.
Using (\ref{Galois group is homology}) to interpret $H_1(M, \ZZ)$ as the group of deck transformations of $\tilde M^{\rm ab}$, let
$$H := \bigcup_{\substack{\rho \in H_1(M, \ZZ) \\ \langle \alpha, \rho\rangle > t - t_0}} \rho(F).$$
For every $x \in H$, there exists $\rho \in H_1(M, \ZZ)$ and $y \in F$, $x = \rho(y)$, and then 
$$u^{\rm ab}(x) = u^{\rm ab}(y) + \langle \alpha, \rho\rangle > t_0 + t - t_0 = t$$
so $H \subseteq \{u^{\rm ab} > t\}$.

Since $H$ is the set of translations of the connected fundamental domain $F$ by a half-space in the deck group, $H$ is connected.
But $\{u^{\rm ab} > t\}$ is disconnected, so there must be a connected component $X$ of $\{u^{\rm ab} > t\}$ which is disjoint from $H$.
For any $\rho \in H_1(M, \ZZ)$ such that $\langle \alpha, \rho\rangle > 0$, $\rho$ sends $\{u^{\rm ab} > t\}$ into itself, since for every $x \in \{u^{\rm ab} > t\}$,
$$u^{\rm ab}(\rho(x)) = u^{\rm ab}(x) + \langle \alpha, \rho\rangle > t + 0 = t.$$
In particular, $\rho$ sends $X$ into a component of $\{u^{\rm ab} > t\}$.
Thus there are two cases to consider:
\begin{prfenum}
\item There exists $\rho \in H_1(M, \ZZ)$ such that $\langle \alpha, \rho\rangle > 0$, but $\rho(X) \subseteq X$. \label{connectivity:intoself}
\item For every $\rho \in H_1(M, \ZZ)$ such that $\langle \alpha, \rho\rangle > 0$, $\rho(X)$ is a subset of a component of $\{u^{\rm ab} > t\}$ which is not $X$. \label{connectivity:notintoself}
\end{prfenum}

In case \ref{connectivity:intoself}, there exists $x \in X$ and $\theta \in H_1(M, \ZZ)$ such that $\theta(x) \in F$; then, for $m$ large, 
$$\langle \alpha, m\rho - \theta\rangle > t - t_0,$$
so $m\rho(x) \in H$.
Therefore $m\rho(X)$ meets $H$, so $X$ meets $H$, a contradiction.

In case \ref{connectivity:notintoself}, let $\hat M$ be the minimal covering space on which $u$ descends to a function $\hat u: \hat M \to \RR$, thus $\Gal(\hat M, M) = \Gamma/\ker(\alpha)$.
Then $\hat u$ has least gradient, and $X$ descends to a component $\hat X$ of $\{\hat u > y\}$.
Then the projection $\psi: \hat M \to M$ restricts to an \emph{injective} map $\hat X \to M$.
Indeed, if $x_1, x_2 \in \hat X$ and $\psi(x_1) = \psi(x_2)$, then there exists $\rho \in H_1(M, \ZZ)/\ker(\alpha)$ such that $\psi(x_1) = x_2$.
If $\rho$ is nonzero, then after switching the roles of $x_1, x_2$ as necessary, we may assume that $\rho$ is represented by some $\overline \rho \in H_1(M, \ZZ)$ such that $\langle \alpha, \overline \rho\rangle > 0$, a contradiction.

By a straightforward generalization of \cite[Theorem 1]{BOMBIERI1969}, $\hat X$ is perimeter-minimizing.
If $\hat X$ is bounded, then $\partial \hat X$ is competing with the empty set and hence is empty, a contradiction; so $\hat X$ is unbounded and therefore has infinite volume by Lemma \ref{unbounded implies infinite measure}.
But $\psi$ is an isometry, so $\psi_*(\hat X)$ is an infinite-volume subset of the closed manifold $M$, a contradiction.
\end{proof}

\begin{lemma}[{\cite{Auer12}}]\label{abelian cover connected}
Let $u$ be an $\alpha$-equivariant function of least gradient on $\tilde M$, and let $\mathscr G$ be a set of curves in $\tilde M^{\rm ab}$ which spans $H_1(\tilde M^{\rm ab}, \RR)$.
If $\partial \{u^{\rm ab} > t\}$ misses every curve in $\mathscr G$, then $\partial \{u^{\rm ab} > t\}$ is connected.
\end{lemma}
\begin{proof}
We reason by contrapositive.
Let $N_1, N_2$ be two distinct components of $\partial \{u^{\rm ab} > t\}$.
By Lemma \ref{superlevel sets are connected} (and the analogous result for sublevel sets), $\tilde M^{\rm ab} \setminus \partial \{u^{\rm ab} > t\}$ has two components $E_1, E_2$.
We construct a curve $\gamma$, transverse to $N_1$, which starts at a point $x \in N_1$, passes through $E_1$, crosses $N_2$ into $E_2$, and then returns to $x$.
In particular $\gamma$ meets $N_1$ at a single point, so their intersection number $[\gamma] \cdot [N_1] = 1$ (possibly after reorienting).
Therefore $[\gamma]$ is a nontrivial class in $H_1(\tilde M^{\rm ab}, \RR)$.
\end{proof}

\begin{proof}[Proof of Theorem \ref{Auer Bangert thm}(\ref{ABt2})]
We prove the contrapositive.
If $B$ is not strictly convex, then there exists $\rho \in \partial B^*$ such that $\rho^*$ is not singleton.
In particular, there are two distinct extreme points $\alpha, \beta \in \mathcal E(\rho^*)$, and by Corollary \ref{extreme points are indecomposable}, we can find distinct ergodic measured oriented sublaminations $\kappa_\alpha, \kappa_\beta$ of $\lambda_\rho$.
Let $u_\alpha, u_\beta$ be primitives of the Ruelle-Sullivan currents on $\tilde M$; by equivariance, they drop to functions $u^{\rm ab}_\alpha, u^{\rm ab}_\beta$ on the universal abelian cover $\tilde M^{\rm ab}$.

There must exist leaves $N_\alpha$ of $\kappa_\alpha$, and $N_\beta$ of $\kappa_\beta$, which are distinct.
If this is not true, then both $\kappa_\alpha, \kappa_\beta$ are the same closed hypersurface, and in particular $\alpha = \beta$, a contradiction.
In particular, by adding constants to $u_\alpha$ and $u_\beta$, we may assume that $\partial \{u_\alpha > 0\}$ and $\partial \{u_\beta > 0\}$ descend to distinct leaves of the covering lamination $\tilde \lambda^{\rm ab}_\rho$.
As sets, $\partial \{u_\alpha > 0\}$ and $\partial \{u_\beta > 0\}$ are boundaries and therefore are closed; they are also disjoint, since they are distinct leaves of the same lamination.
Therefore they are separated by open sets.

Since $\star |\dif u_\alpha|$ and $\star |\dif u_\beta|$ are elements of $\mathcal M(\lambda_\rho)$, so is their mean, which can be expressed as $\star |\dif u|$ where $u := (u_\alpha + u_\beta)/2$.
In particular, $u$ has least gradient, and
$$\partial \{u > 0\} = \partial \{u_\alpha > 0\} \cup \partial \{u_\beta > 0\}$$
and since the right hand side is separated by open sets, $\partial \{u > 0\}$ is disconnected. 
So by Lemma \ref{abelian cover connected}, $H_1(\tilde M^{\rm ab}, \RR)$ is nonzero.
\end{proof}

\section{Epilogue}
\subsection{The earthquake norm}\label{Teichmuller conjectures}
The picture which seems to be emerging is that the stable norm is highly analogous to the earthquake norm on the cotangent space $T_\sigma^* \mathscr T_g$ to the Teichm\"uller space of a closed hyperbolic surface $(\Sigma_g, \sigma)$ of genus $g$.
The starting point for this observation is the earthquake theorem, \cite{Kerckhoff83}, which asserts that one can identify each $\alpha \in T_\sigma^* \mathscr T_g$ uniquely with a measured geodesic lamination in $(\Sigma_g, \sigma)$; the \dfn{earthquake norm} $\|\alpha\|_{\rm eq}$ is the mass of the corresponding measure.
It follows from \cite[Theorem 1.6]{daskalopoulos2023} that for each geodesic lamination $\lambda$ in $(\Sigma, \sigma)$, the map $\mathcal M(\lambda) \to T_\sigma^* \mathscr T_g$ is affine.
The dual norm to the earthquake norm is the stretch norm, which is given by infinitesimal minimizing Lipschitz maps, just as the costable norm is given by calibrations.

The earthquake norm satisfies an analogue of Theorem \ref{Auer Bangert thm}(\ref{ABt1}).
(Theorem \ref{Auer Bangert thm}(\ref{ABt2}) holds vacuously, since $\pi_1(\Sigma_g)$ is nonabelian and free.)

\begin{theorem}[{\cite[Theorem 6.1]{huang2024earthquakemetricteichmullerspace}}]\label{maximal flats are Lagrangian}
Let $\alpha, \beta \in T_\sigma^* \mathscr T_g$ satisfy $\|\alpha\|_{\rm eq} = \|\beta\|_{\rm eq} = 1$.
The following are equivalent:
\begin{enumerate}
\item $\alpha, \beta$ are contained in the same maximal flat of the earthquake unit sphere.
\item As measured geodesic laminations, $\alpha, \beta$ do not intersect transversely.
\end{enumerate}
\end{theorem}

The intersection product on measured geodesic laminations on $(\Sigma_g, \sigma)$ corresponds to the Weil-Petersson $2$-form; in particular, it is symplectic.
So by Theorem \ref{maximal flats are Lagrangian}, every earthquake flat $F$ is contained in a $3g - 3$-dimensional subspace of $T_\sigma^* \mathscr T_g$.
Therefore the earthquake norm also satisfies an analogue of Theorem \ref{vertex counting}:

\begin{corollary}\label{crly: earthquake vertex counting}
Let $F \subset T_\sigma^* \mathscr T_g$ be a maximal flat of the earthquake unit sphere.
Then:
\begin{enumerate}
\item There exists a geodesic lamination $\lambda$ such that $F = \mathcal M(\lambda)$.
\item $F$ is a convex polytope with at most $3g - 3$ vertices.
\end{enumerate}
\end{corollary}
\begin{proof}
Let $(\alpha_i)$ be a dense sequence in $F$.
By Theorem \ref{maximal flats are Lagrangian}, if we let $\lambda_i := \bigcup_{j \leq i} \supp \alpha_j$, then $\lambda_i$ is (the support of) a geodesic lamination such that $\lambda_i \subseteq \lambda_{i + 1}$.
Taking the limit (say, in the Vietoris topology), $\lambda_i$ converges to a lamination $\lambda$ for which $\alpha_i \in \mathcal M(\lambda)$.
Since $\mathcal M(\lambda)$ is compact, $F \subseteq \mathcal M(\lambda)$, and since $F$ is maximal, $\mathcal M(\lambda) \subseteq F$.
By an easy generalization of Theorem \ref{ergodic decomposition thm}, $\mathcal E(F)$ is the set of ergodic measures on $\mathcal M(\lambda)$, so it is linearly independent.
But $\mathcal E(F)$ is contained in a $3g - 3$-dimensional vector space, so $\card \mathcal E(F) \leq 3g - 3$.
\end{proof} 

\begin{conjecture}\label{earthquake conjecture}
For every vector $v$ in the stretch unit sphere of $T_\sigma \mathscr T_g$ there exists a hyperbolic structure $\tau \in \mathscr T_g$ such that $v^*$ is the canonical lamination maximally stretched by the homotopy class of the identity $(\Sigma_g, \sigma) \to (\Sigma_g, \tau)$.
\end{conjecture}

Conjecture \ref{earthquake conjecture} seems quite likely to hold in view of Corollaries \ref{extreme points are indecomposable} and \ref{crly: earthquake vertex counting}.
A natural attempt to prove Conjecture \ref{earthquake conjecture} is to show that there is a diffeomorphism $\exp_\sigma: T_\sigma \mathscr T_g \to \mathscr T_g$ which maps every ray emanating from $0$ in $T_\sigma \mathscr T_g$ to a geodesic ray in $\mathscr T_g$ with respect to Thurston's stretch metric, such that $v^* \subseteq \mathcal M(\lambda_{\sigma, \exp_\sigma(v)})$.
We refer to Pan and Wolf, \cite{pan2022raystructuresteichmullerspace}, for more discussion of exponential maps for Thurston's stretch metric.

\subsection{Higher dimension and codimension}\label{sec: high dimension}
In this paper we only considered cohomology classes in $H^{d - 1}(M, \RR)$ where the manifold $M$ had dimension $d \leq 7$.
However, the definition of the costable norm makes sense in any cohomology group $H^k(M, \RR)$ and so it is natural to ask if Theorem \ref{existence of calibrated lam} holds for $k \leq d - 2$ or $d \geq 8$.

A version of Theorem \ref{existence of calibrated lam} is available when $k = 1$.
Let $\rho \in H^1(M, \RR)$ and $\|\rho\|_\infty = 1$.
By thinking of $\rho$ as a representation $\pi_1(M) \to \RR$, we can identify calibrations in $\rho$ with functions $v$ on $\tilde M$ such that $\Lip(v) = 1$.
From this, and the proof of \cite[Lemma 5.2]{Gu_ritaud_2017}, it follows that there is a unique largest geodesic lamination $\lambda_\rho$ which is calibrated by every calibration in $\rho$.
It seems plausible that one can generalize arguments of \cite[\S6]{daskalopoulos2020transverse}, to establish that $\rho^*$ is the set of homology classes of transverse measures to sublaminations of $\lambda_\rho$, and show that the rational extreme points of $\rho^*$ correspond to closed leaves in $\lambda_\rho$.
However, if $d \geq 3$ then the codimension of $\lambda_\rho$ is too large to hope that any of the intersection-theoretic results of this paper generalize to the $k = 1$ case.

If $2 \leq k \leq d - 2$, then there can be no generalization of Theorem \ref{existence of calibrated lam} to $H^k(M, \RR)$, since the submanifolds calibrated by a class $\rho \in H^k(M, \RR)$ can intersect each other.
For example, let $M = \PP^2_\CC$, let $\rho$ be the K\"ahler class of $\PP^2_\CC$, and let $F$ be the K\"ahler form of $\PP^2_\CC$, which is a calibration by Wirtinger's inequality.
Then every projective algebraic curve in $\PP^2_\CC$ is $F$-calibrated, and so is a leaf of the putative canonical lamination calibrated by $\rho$!

The construction of the canonical calibrated lamination heavily relied on the regularity theory of stable minimal hypersurfaces \cite{Schoen81}, which is only valid when $d \leq 7$.
We outline a possible approach to generalizing this paper to $d \geq 8$.
Let $F$ be a calibration, $\rho := [F]$, $N$ be an $F$-calibrated hypersurface, and $x \in N$.
After throwing away a set of codimension $\geq 8$, we may assume that $x$ is not a singular point of $N$, so that $|\Two_N|$ is bounded near $x$.
Let $N'$ be another $F$-calibrated hypersurface such that $x$ is not a singular point of $N'$.
It should follow from Simon's maximum principle \cite{Simon87} and the proof of Lemma \ref{calibrated implies disjoint} that there exists $\delta > 0$ such that $B(x, \delta) \cap N \cap N' = \emptyset$.

The natural next step in the construction of $\lambda_\rho$ would be to assume towards contradiction that $N'$ oscillates wildly at some point $y \in N' \cap B(x, \delta)$.
By elliptic regularity, and the fact that $N, N'$ are disjoint, one would like to conclude that $N$ also oscillates wildly, contradicting the bound on $|\Two_N|$.
If that argument worked, then we would conclude that $|\Two_{N'}|$ is bounded near $y$.
By Allard's regularity theorem \cite[Theorem 7.8]{WhiteNotes}, that argument \emph{almost} works, but actually only shows that if $|\Two_{N'}|$ is very large at some point close to $y$, then on coarse scales, $N'$ looks like a current of higher multiplicity at $y$.
This is the same mechanism by which a stack of catenoids (Example \ref{ex: stack of catenoids}) fails to be a lamination.

It is possible that stacks of catenoids constitute the worst-case scenario.
Motivated by this, we make two bold conjectures.

\begin{conjecture}
Let $\mathcal S$ be a set of disjoint minimal hypersurfaces in $M$.
Then there is a closed set $Z$ such that $\dim Z \leq d - 3$, and for every $U \Subset M$ there is a Lipschitz lamination $\lambda$ in $U$ such that $\{N \cap U: N \in \mathcal S\}$ is dense in the set of leaves of $\lambda$.
\end{conjecture}

\begin{conjecture}
Let $d \geq 8$, let $\rho \in H^{d - 1}(M, \RR)$, assume that $\|\rho\|_\infty = 1$, and let $\mathcal S$ be the set of minimal hypersurfaces which are calibrated by every calibration in $\rho$.
Then there is a closed set $Z$ such that $\dim Z \leq d - 3$ and, for every $U \Subset M \setminus Z$, there is a Lipschitz lamination $\lambda$ in $U$ such that $\{N \cap U: N \in \mathcal S\}$ is dense in the set of leaves of $\lambda$.
\end{conjecture}

The results of \S\ref{canonical structure} should not be able to detect the presence of a codimension $\geq 3$ singular set, so we expect them to generalize ceteris paribus to $d \geq 8$.

\printbibliography

\end{document}